\documentclass{article}
\usepackage{fullpage}
\usepackage{amsmath, amssymb, amsthm}
\usepackage{mathtools,stmaryrd}
\usepackage{ifthen}
\usepackage[capitalise]{cleveref}
\usepackage{todonotes}
\usepackage{url}
\usepackage{subcaption}
\usepackage{physics2}
\usephysicsmodule{ab}
\usepackage{tikz}
\usetikzlibrary{calc,arrows,positioning}

\theoremstyle{definition}
\newtheorem{theorem}{Theorem}[section]

\newtheorem{lemma}[theorem]{Lemma}

\newtheorem{proposition}[theorem]{Proposition}
\newtheorem{definition}[theorem]{Definition}

\newtheorem{example}[theorem]{Example}

\newtheorem{remark}[theorem]{Remark}

\def\cF{{\cal F}}

\def\cM{{\cal M}}

\def\cR{{\cal R}}

\def\cU{{\cal U}}

\def\cX{{\cal X}}
\def\cY{{\cal Y}}
\def\cZ{{\cal Z}}
\def\bbZ{{\mathbb Z}}
\def\bfC{{\mathbf C}}

\def\bfK{{\mathbf K}}

\def\bfn{{\mathbf n}}

\newcommand{\timesdots}{\times\dots\times}

\DeclareMathOperator{\im}{im}
\DeclareMathOperator{\rank}{rank}
\newcommand{\Hom}{\mathrm{Hom}}

\DeclareMathOperator{\Tor}{Tor}

\newcommand{\Mor}{\mathrm{Mor}}
\newcommand{\Term}{\mathrm{Term}}

\newcommand{\Id}{\mathrm{id}}
\newcommand{\Ob}{\mathrm{Ob}}
\newcommand{\Op}{\mathrm{op}}

\newcommand{\cod}{\mathrm{cod}}
\newcommand{\catab}{\mathbf{Ab}}
\newcommand{\catset}{\mathbf{Set}}

\newcommand{\catfam}{\mathbf{Fam}}

\newcommand{\cattheories}{\mathbf{Law}}

\newcommand{\In}{\mathrm{in}}
\newcommand{\Pos}{\mathrm{Pos}}

\newcommand{\Var}{\mathrm{Var}}

\newcommand{\varlaw}{\mathbf{L}}

\newcommand{\varsorts}{S}

\newcommand{\varobj}[1][0]{%
  {\ifthenelse{\equal{#1}{0}}{\vec{X}}{%
    \ifthenelse{\equal{#1}{1}}{\vec{Y}}{\vec{Z}}%
  }}%
}
\newcommand{\varsort}[1][0]{%
  {\ifthenelse{\equal{#1}{0}}{{X}}{%
    \ifthenelse{\equal{#1}{1}}{{Y}}{{Z}}%
  }}%
}

\newcommand{\varsign}{\Sigma}

\newcommand{\legensub}{\preceq_{\mathrm{gs}}}
\newcommand{\ltgensub}{\prec_{\mathrm{gs}}}

\newcommand{\kahlerdiffs}{\Omega}

\makeatletter
\newcommand*{\relrelbarsep}{.386ex}
\newcommand*{\relrelbar}{%
  \mathrel{%
    \mathpalette\@relrelbar\relrelbarsep
  }%
}
\newcommand*{\@relrelbar}[2]{%
  \raise#2\hbox to 0pt{$\m@th#1\relbar$\hss}%
  \lower#2\hbox{$\m@th#1\relbar$}%
}
\providecommand*{\rightrightarrowsfill@}{%
  \arrowfill@\relrelbar\relrelbar\rightrightarrows
}
\providecommand*{\leftleftarrowsfill@}{%
  \arrowfill@\leftleftarrows\relrelbar\relrelbar
}
\providecommand*{\xrightrightarrows}[2][]{%
  \ext@arrow 0359\rightrightarrowsfill@{#1}{#2}%
}
\providecommand*{\xleftleftarrows}[2][]{%
  \ext@arrow 3095\leftleftarrowsfill@{#1}{#2}%
}
\makeatother
\newcommand{\Cr}{\mathrm{Cr}}
\newcommand{\syn}[1]{\mathrm{Syn}\ab(#1)}
\newcommand{\bfp}{\mathbf{p}}

\title{Anick Resolution for Lawvere Theories from\\ Algebraic Discrete Morse Theory}
\author{Mirai Ikebuchi}

\begin{document}

\maketitle
\begin{abstract}
Inspired by Brown's collapsing method (or discrete Morse theory) to obtain a free resolution of $\bbZ$ over the monoid ring $\bbZ M$, we apply algebraic discrete Morse theory to compute the homology groups of Lawvere theories, which is defined as Tor of a certain module. We reinterpret known partial free resolutions arising from complete term rewriting systems in terms of collapsing of the normalized bar resolution. This perspective yields homological inequalities that bound the number of equational axioms in presentations and recovers classical results, such as lower bounds for group axiomatizations. Our main contribution is to extend these resolutions to higher dimensions.
\end{abstract}

\section{Introduction}
In this paper, we apply algebraic discrete Morse theory to compute the (co)homology groups of Lawvere theories.
Before stating our results, we explain the background and motivation of this work.

For a monoid (or a group) \(M\) with a presentation \((\Gamma, R)\), there exists a partial free resolution of \(\bbZ\) as a left module over the monoid ring \(\bbZ M\):
\[
\bigoplus_R \bbZ M \to \bigoplus_\Gamma \bbZ M \to \bbZ M \to \bbZ \to 0.
\]
This implies the inequality
\begin{equation}\label{eqn:monoid_ineq}
\# R - \# \Gamma \ge s(\Tor_2^{\bbZ M}(\bbZ, \bbZ)) - \rank(\Tor_1^{\bbZ M}(\bbZ, \bbZ))
\end{equation}
where \(s(H)\) is the minimum number of generators of an abelian group \(H\), and \(\rank(H)\) is the torsion-free rank of \(H\).

Note that \(\bbZ M\) can be written as \(\bbZ\Gamma^*/I\) where \(\Gamma^*\) is the free monoid generated by \(\Gamma\) and \(I\) is the ideal of \(\bbZ \Gamma^*\) generated by \(U=\{u - v \mid (u,v) \in R\}\).
In \cite{anick86}, Anick showed that, given a minimal (noncommutative) Gr\"obner basis \(G\) of \(I\), there is a resolution of \(\bbZ\)
\[
\dots \to F_n \to \dots \to F_2 \to F_1 \to F_0 \to \bbZ \to 0
\]
where each \(F_n\) is generated by \emph{n-chains}, which are certain \(n\)-tuples of words over \(\Gamma\), and there are one-to-one correspondences between the set of \(0\)-chains and a singleton set, and the set of \(1\)-chains and \(\Gamma\), and the set of \(2\)-chains and \(G\).
If \(\Gamma\) and \(G\) are finite, for each \(n\), the set of \(n\)-chains is known to be finite, so, \(F_n\) is finitely generated.
Brown observed that if \(U\) is a minimal Gr\"obner basis, it provides a scheme to collapse cells of the normalized bar resolution without changing the homotopy type, and the collapsed complex induces Anick's resolution \cite{brown92}.
Brown's collapsing method is now known as \emph{discrete Morse theory} for CW complexes \cite{FORMAN98} and generalized to chain complexes of modules over arbitrary rings \cite{joellenbeck05,skoldberg06}.
From this viewpoint, we can call \eqref{eqn:monoid_ineq} a Morse inequality.

We consider an analogue of this story for Lawvere theories.
A (single-sorted) Lawvere theory is a small category with finite products such that every object is a finite power of a distinguished object.
The notion of Lawvere theories gives an abstract framework for universal algebra \cite{lawvere63},
and every Lawvere theory can be ``presented'' by \((\varsign, E)\) where \(\varsign\) is a set of operation symbols and \(E\) is a set of equational axioms built from symbols in \(\varsign\) and symbols for variables.
For two sets \(E,E'\) of equational axioms over \(\varsign\), if \(E\) and \(E'\) are logically equivalent, then \((\varsign, E)\) and \((\varsign, E')\) present isomorphic Lawvere theories.
(Details are given in \cref{sec:lawvere}.)

In \cite{jp06}, Jibladze and Pirashvili studied cohomology theory for Lawvere theories.
They showed that the two cohomology theories of Lawvere theories coincide: the Quillen cohomology, which is defined for general algebraic structures \cite{quillen2006homotopical}, and the Baues--Wirsching cohomology, which is defined for small categories \cite{baues85}, with a certain class of coefficients.
Moreover, they showed that for any Lawvere theory \(\varlaw\), there exists a ringoid \(\cU_\varlaw\) and a left \(\cU_\varlaw\)-module \(\kahlerdiffs_\varlaw\) such that the Quillen and Baues--Wirsching cohomology of \(\varlaw\) coincides with Ext of \(\kahlerdiffs_\varlaw\) over left \(\cU_\varlaw\)-modules.

After a decade, Malbos and Mimram showed that, if \(\varlaw\) has a presentation \((\varsign,E)\) such that \(E\) is a \emph{complete term rewriting system}, there is a partial free resolution
\begin{equation}\label{eqn:partial_res}
\cF_3 \to \cF_2 \to \cF_1 \to \cF_0 \to \kahlerdiffs_\varlaw \to 0
\end{equation}
where each \(\cF_i\) is free, \(\cF_1\) is generated by \(\varsign\), \(\cF_2\) is generated by \(E\), and \(\cF_3\) is generated by a finite set if \(\varsign,E\) are finite \cite{mm16}.

Also, the author showed \cite{ikebuchi22} that under a condition on the \emph{degree} \(d\) of \(E\), there exists a right \(\cU_\varlaw\)-module \(\cZ_d\) such that we have the inequality
\begin{equation}\label{eqn:main_inequality}
\# E - \# \varsign + 1 \ge s(H_2(\varlaw; \cZ_d)) - \rank(H_1(\varlaw; \cZ_d)) + \rank(H_0(\varlaw; \cZ_d)) \quad (H_n(\varlaw; \cZ_d) = \Tor_n^{\cU_\varlaw}(\cZ_d, \kahlerdiffs_\varlaw)).
\end{equation}
This inequality can be used to bound by below the sizes of sets of equational axioms.
For example, consider the Lawvere theory \(\varlaw = \varlaw_{\mathrm{grp}}\) presented by the set \(\varsign_0 = \{\_ \cdot\_, e, \_^{-1}\}\) of symbols and the set \(E_0\) of group axioms
\begin{align*}
x \cdot e \approx x, \quad e \cdot x \approx x, \quad x \cdot x^{-1} \approx e, \quad x^{-1} \cdot x \approx e,\quad (x \cdot y) \cdot z \approx x \cdot (y \cdot z).
\end{align*}
It is known that there is a complete term rewriting system \(E_1\) over \(\varsign_0\) equivalent to \(E_0\) \cite{Knuth83}.
By computing the lower dimensional homology groups of \(\varlaw_\mathrm{grp}\) using the resolution \eqref{eqn:partial_res}, we obtain (RHS of \eqref{eqn:main_inequality}) = 0.
Since the homology groups do not depend on the choice of presentations, we conclude that \(\# E \ge \# \varsign_0 - 1 = 2\) holds for any set \(E\) of equational axioms over \(\varsign_0\) equivalent to \(E_0\).
This is a new proof that there is no single equational axiom for groups over \(\varsign_0\), which is a classical result originally proved in some ways specific to the theory of groups \cite{t68,n86,k92}.

The goal of this paper is to extend the resolution \eqref{eqn:partial_res} to higher dimensions using algebraic discrete Morse theory.
The outline is as follows.
\begin{itemize}
	\item In Section 2, we review the definitions and basic properties of Lawvere theories and equational presentations.
	\item In Section 3, after an introduction of ringoids and modules over ringoids, we explain algebraic discrete Morse theory over ringoids.
	\item In Section 4, we first define the enveloping ringoid \(\cU_\varlaw\) of a Lawvere theory \(\varlaw\) and the \(\cU_\varlaw\)-module of K\"ahler differentials \(\kahlerdiffs_\varlaw\) following Jibladze and Pirashvili \cite{jp06}. Then, we construct the unnormalized bar resolution of \(\kahlerdiffs\), and define the normalized bar resolution using algebraic discrete Morse theory.
	\item In Section 5, we apply algebraic discrete Morse theory to the normalized bar resolution to obtain Anick resolution for Lawvere theories, which is our main result.
	\item In Section 6, we show Morse inequalities for Lawvere theories as an application of Anick resolution.
\end{itemize}

\section{Lawvere theories}\label{sec:lawvere}
%\begin{definition}
%Let $\varlaw$ be a category with finite products.
%A \emph{model} of $\varlaw$ is a product-preserving functor $\varlaw \to \catset$, and a \emph{morphism of models} is a natural transformation.
%We write $\Alg\varlaw$ for the category of models of $\varlaw$.
%\end{definition}
%
%\begin{definition}
%If a category $\varcat$ is equivalent to $\Alg\varlaw$ for some Lawvere theory $\varlaw$, we say that $\varcat$ is \emph{algebraic}.
%\end{definition}

%Let $\bfS^{\Op}$ be a full subcategory of $\catset$ with objects $\bfn = \ab\{1,\dots,n\}$ for $n=0,1,\dots$.
%(This $\bfS$ is called the \emph{Lawvere theory of sets}.)
%
Let $\bfn$ be the set $\ab\{1,\dots,n\}$ for $n=0,1,\dots$.
For a set $\varsorts$, let $\catfam_\varsorts^\Op$ be the full subcategory of $\catset/\varsorts$ with objects $f: \bfn \to \varsorts$ for $n=0,1,\dots$.
Note that $\catfam_\varsorts^\Op$ has finite coproducts $f_1 + f_2 : \bfn_1 + \bfn_2 \to \varsorts$ for $f_i : \bfn_i \to \varsorts$ ($i=1,2$), so $\catfam_\varsorts = (\catfam_\varsorts^\Op)^\Op$ has finite products.
We write $\varsorts^*$ for $\Ob(\catfam_\varsorts)$.

Note that any object $f : \bfn \to S$ can be thought of as a string $X_1\dots X_n$ over $S$ for $f(i) = X_i$, and the product of $X_1\dots X_{n_1}$ and $Y_1\dots Y_{n_2}$ is the concatenation $X_1\dots X_{n_1} Y_1\dots Y_{n_2}$.
A morphism $X_1\dots X_n \to Y_1\dots Y_m$ is a function $u : \mathbf{m} \to \mathbf{n}$ such that $X_{u(i)} = Y_i$ for each $i\in \mathbf{m}$.
Also, we can check that $X_1\dots X_n$ is isomorphic to any permutation $X_{\sigma(1)}\dots X_{\sigma(n)}$ where $\sigma : \bfn \to \bfn$ is a bijection.

\begin{definition}
\begin{itemize}
	\item For a set $\varsorts$, an $\varsorts$-\emph{sorted Lawvere theory} is a small category $\varlaw$ that has finite products together with a morphism $\iota : \catfam_\varsorts \to \varlaw$ of Lawvere theories such that the objects of $\varlaw$ are functions $\bfn \to \varsorts$ and $\iota$ is identity on objects.
We often call $\varlaw$ an $\varsorts$-sorted Lawvere theory without mentioning $\iota$.
	\item A morphism between $\varsorts$-sorted Lawvere theories $\iota : \catfam_\varsorts \to \varlaw$, $\iota' : \catfam_\varsorts \to \varlaw'$ is a functor $F : \varlaw \to \varlaw'$ that identity on objects, preserves products, and satisfy $F \circ \iota = \iota'$.
	\item The category of $\varsorts$-sorted Lawvere theories is denoted by $\cattheories_\varsorts$.
\end{itemize}
\end{definition}

Notice that giving \(\iota : \catfam_\varsorts \to \varlaw\) is the same as choosing, for each \(\vec X = X_1 \timesdots X_n \in S^*\), a product diagram in \(\varlaw\) whose factors are \(X_1,\dots,X_n\).
So, a category $\varlaw$ is an $\varsorts$-sorted Lawvere theory if $\Ob(\varlaw) = \Ob(\catfam_\varsorts)$ and if 
a product diagram \(\ab\{\pi^{\vec X}_i : \vec X\to C_i \}_{i \in \{1,\dots,n\}}\) in \(\varlaw\) is chosen for each \(\vec X = X_1 \timesdots X_k \in S^*\).
When we work on an \(\varsorts\)-sorted Lawvere theory \(\varlaw\), \(\pi^{\vec X}_i\) or just \(\pi_i\) denotes the projection in the chosen product diagram.
A functor between $\varsorts$-sorted Lawvere theories $\varlaw \to \varlaw'$ is a morphism of $\varsorts$-sorted Lawvere theories if it is a morphism of Lawvere theories and maps the chosen product diagrams in $\varlaw$ to those in $\varlaw'$.

\begin{remark}
For objects in an $S$-sorted Lawvere theory, we write $X,Y,\dots$ when they are sorts (elements in $S$, not $S^*$) and write $\vec X,\vec Y,\dots$ for general elements in $S^*$.
Also, for morphisms, we write $f,g,\dots$ when their codomains are sorts (i.e., $\cod(f),\cod(g)\in S$) and write $\vec f,\vec g,\dots$ for general morphisms.
\end{remark}

We now introduce the notion of equational presentations of an \(\varsorts\)-sorted Lawvere theory.

\begin{definition}
Let $\varsorts$ be a set of sorts and $V_X$ be an infinite set of \emph{variable symbols of sort $X$} for each sort $X \in \varsorts$.
Let $V = \coprod_{X \in \varsorts} V_X$.
\begin{itemize}
	\item An $\varsorts$-\emph{sorted signature} is a set $\varsign$ (of \emph{operation symbols}) together with a function $\alpha : \varsign \to \varsorts^* \times \varsorts$.
If $\alpha(f) = (X_1\dots X_n,X)$ for $f \in \varsign$, we write $f : X_1\timesdots X_n \to X$.
	\item A \emph{term of sort} $X$ over $\varsign$ and $V$ is defined inductively as follows. (i) Any variable symbol $x \in V$ of sort $X$ is a term of sort $X$. (ii) If $f : X_1\timesdots X_n \to X$ and $t_1,\dots,t_n$ are terms of sorts $X_1,\dots,X_n$, respectively, then the formal expression $f(t_1,\dots,t_n)$ is a term of sort $X$.
	\item $\Term_\varsign(V)$ denotes the set of all terms over $\varsign$ and $V$, and \(\Term_\varsign^X(V)\) denotes the set of all terms of sort \(X\) over \(\varsign\) and \(V\).
	\item For a term \(t\), \(\Var(t)\) is the set of variables occurring in \(t\).
	\item A finite list of distinct variables $x_1,\dots ,x_n$ is called a \emph{context} (or \emph{sorting context}). We often write a context as $x_1:X_1,\dots,x_n:X_n$ where $X_i$ is a sort of $x_i$.
	\item For a term $t$ and a context $x_1:X_1,\dots,x_n:X_n$ such that $\Var(t)\subseteq \{x_1,\dots,x_n\}$, we call the formal expression $x_1:X_1,\dots,x_n:X_n \vdash t$ a \emph{term-in-context}.
	\item For a term \(t\) whose variables are \(x_1:\varsort_1,\dots,x_n:\varsort_n\) and terms \(t_1,\dots,t_n\) of sorts \(\varsort_1,\dots,\varsort_n\), respectively, we write \(t[t_1/x_1,\dots,t_n/x_n]\) for the term obtained from \(t\) by replacing each variable \(x_i\) with \(t_i\).
	\item For a sort \(\varsort \in \varsorts\), a term \(C \in \Term_\varsign(V\cup\{\square_{\varsort}\})\) is called a \emph{rewrite context} if \(\square_{\varsort}\) occurs exactly once in \(C\) where \(\square_{\varsort}\) is a fresh symbol called the \emph{hole} and has sort \(\varsort\).
	\item For a rewrite context \(C\) with hole \(\square_X\) of sort \(X\) and a term \(t\) of sort \(X\), we write \(C[t]\) for the term obtained from \(C\) by replacing \(\square_X\) with \(t\).
	\item An \emph{equation} is a pair $(t_1,t_2)$ of two terms of a same sort. An equation is written as $t_1 \approx t_2$.
	\item A pair $(\varsign, E)$ of a signature and a set of equations is called an \emph{($\varsorts$-sorted) equational presentation}.
\end{itemize}
\end{definition}
\begin{example}[Abelian groups]\label{ex:abelian}
Let $\varsorts$ be a singleton set $\{X\}$ and $\varsign$ be $\{0,-,+\}$ with $\alpha(0) = (\epsilon, X)$, $\alpha(-) = (X,X)$, $\alpha(+) = (XX, X)$.
We write $t_1+t_2$ for the term $+(t_1,t_2)$.
The following set $E$ of equations presents the theory of abelian groups:
\[
x + 0 \approx x,\quad x+ (-x) \approx 0, \quad (x_1+ x_2)+ x_3 \approx x_1+ (x_2+ x_3), \quad x_1 + x_2 \approx x_2 + x_1
\]
where $x_1,x_2,x_3 \in V = V_X$.
\end{example}
\begin{example}[Left modules over a monoid]\label{ex:module}
Let $\varsorts' = \{X,Y\}$ and $\varsign' = \{0,-,+,1,\circ,\cdot\}$ with $\alpha'(1) = (\epsilon,Y)$ $\alpha'(\circ) = (YY,Y)$, $\alpha'(\cdot) = (YX,X)$, and $\alpha'$ is defined in the same way as in the previous example for $0,-,+$.
We write $t_1\circ t_2$ for $\circ(t_1,t_2)$ and $t_1\cdot t_2$ for the term $\cdot(t_1,t_2)$.
Then, the following set $E'$ of equations together with the equations in the previous example presents the theory of left modules over a monoid:
\begin{align*}
&1 \circ y \approx y,\quad y \circ 1 \approx y,\quad (y_1 \circ y_2)\circ y_3 \approx y_1 \circ (y_2 \circ y_3),\\
&y \cdot 0 \approx 0,\quad 1 \cdot x \approx x,\quad (y_1 \circ y_2)\cdot x \approx y_1\cdot (y_2\cdot x), \quad y\cdot (x_1 + x_2) \approx y\cdot x_1 + y\cdot x_2
\end{align*}
where $x,x_i \in V_X$ and $y,y_i \in V_Y$.
That is, the equations in the first line say that $\circ$ is the monoid multiplication with the unit $1$, and the equations in the second line are the laws for the scalar multiplication $y \cdot x$.
\end{example}

\begin{definition}\label{def:congruence}
We define the equivalence relation \(t \approx_E s\) between terms \(t,s \in \Term_\varsign(V)\) of the same sort as the smallest equivalence relation satisfying the following conditions:
\begin{enumerate}
	\item For any $t \approx s \in E$, $t \approx_E s$.
	\item Let $t$, $s$ be terms, $x_1,\dots,x_n$ be the variables that occur in $t$ or $s$ whose sorts are $X_1,\dots,X_n$, and $t_1,\dots,t_n$ be terms of sorts $X_1,\dots,X_n$, respectively.
If $t\approx_E s$, then $t[t_1/x_1,\dots,t_n/x_n] \approx_E s[t_1/x_1,\dots,t_n/x_n]$.
	\item For any rewrite context \(C\) with hole of sort \(X\) and terms \(t,t'\) of sort \(X\), if \(t \approx_E t'\), then \(C[t] \approx_E C[t']\).
\end{enumerate}
\end{definition}

We extend \(\approx_E\) to term-in-contexts as \(\Gamma \vdash t \approx_E \Delta \vdash s\) iff \(\Gamma = \Delta\) and \(t \approx_E s\).

Let \(\sim_\alpha\) be the equivalence relation between \(m\)-tuples of terms defined as follows: \((t_1,\dots,t_m) \sim_\alpha (t_1',\dots,t_m')\) iff, for \(\{x_1,\dots,x_n\} = \Var(t_1)\cup \dots \cup \Var(t_m)\) and \(\{y_1,\dots,y_k\} = \Var(t_1')\cup \dots \cup \Var(t_m')\), we have \(k = n\) and \(t_i[y_1/x_1,\dots,y_n/x_n] = t_{i}'\) for each \(i=1,\dots,m\).%
\footnote{The symbol \(\alpha\) comes from \emph{alpha equivalence} in theoretical computer science which means that the two expressions are identical up to renaming of variables.}

We define \(\sim_\alpha\) for \(m\)-tuple of term-in-contexts \((\Gamma \vdash t_1,\dots,\Gamma\vdash t_m)\) and \((\Delta\vdash s_1,\dots \Delta\vdash s_m)\) as
\((\Gamma \vdash t_1,\dots,\Gamma \vdash t_m ) \sim_\alpha (\Delta \vdash s_1,\dots,\Delta\vdash s_m) \) iff
\(\Gamma = x_1:X_1,\dots,x_n:X_n\), \(\Delta = y_1:X_1,\dots,y_n:X_n\), and \(t_i[y_1/x_1,\dots,y_n/x_n] = s_i\) for \(i=1,\dots,m\).
Also, we define an equivalence relation \(\approx_{E,\alpha}\) as \((\Gamma \vdash t_1,\dots,\Gamma \vdash t_m) \approx_{E,\alpha} (\Delta \vdash s_1,\dots,\Delta \vdash s_m)\) iff \(\Gamma = x_1:X_1,\dots,x_n:X_n\), \(\Delta = y_1:X_1,\dots,y_n:X_n\), \(t_i[y_1/x_1,\dots,y_n/x_n] \approx_E s_i\) for \(i=1,\dots,m\).
We write \(\Gamma \mid_E (t_1,\dots,t_m)\) for the equivalence class of \((\Gamma \vdash t_1,\dots,\Gamma \vdash t_m)\) with respect to \(\approx_{E,\alpha}\).

\begin{definition}
Given \((\varsign,E)\), we can construct an \(\varsorts\)-sorted Lawvere theory \(\syn{\varsign,E}\) as follows:
\begin{itemize}
	\item \(\Ob(\syn{\varsign,E}) = \varsorts^*\),
	\item \(\Hom_{\syn{\varsign,E}}(X_1\dots X_n, Y_1\dots Y_m)\) is the set of equivalence classes \(\Gamma \mid_E (t_1,\dots,t_n)\) where \(\Gamma = x_1:X_1,\dots,x_n:X_n\) and \(\Gamma \vdash t_i : Y_i\) for \(i=1,\dots,m\),
	\item \((\Delta \mid_E (s_1,\dots,s_l)) \circ (\Gamma \mid_E (t_1,\dots,t_m)) = (\Gamma \mid_E (s_1[t_1/y_1,\dots,t_m/y_m],\dots, s_l[t_1/y_1,\dots,t_m/y_m]))\) for any \(\Delta = y_1:Y_1,\dots,y_m:Y_m\), \(\Delta\vdash s_i : Z_i\), and \(\Gamma\), \(t_i\) as above.
\end{itemize}	
We call \(\syn{\varsign,E}\) the \(\varsorts\)-\emph{sorted Lawvere theory presented by \((\varsign,E)\)}.
We write \(\syn{\varsign}\) for \(\syn{\varsign,\emptyset}\).
\end{definition}

We have the \emph{forgetful} functor $U_{\cattheories_\varsorts}: \cattheories_\varsorts \to \catset^{\varsorts^* \times \varsorts}$ where $\catset^{\varsorts^*\times \varsorts}$ is the functor category from the discrete category $\varsorts^*\times \varsorts$ to $\catset$ defined as
\begin{align*}
U_{\cattheories_\varsorts}(\iota:\catfam_\varsorts \to \varlaw)(X,Y) &= \Hom_{\varlaw}(X,Y)\\
U_{\cattheories_\varsorts}(f : \varlaw \to \varlaw')_{(X,Y)} &= f|_{\Hom_\varlaw(X,Y)} : \Hom_{\varlaw}(X,Y) \to \Hom_{\varlaw'}(X,Y).
\end{align*}

By identifying a signature \(\varsign\) with the functor \(\varsorts^*\times \varsorts \to \catset\), \((X_1\dots X_k, X) \mapsto \{f : X_1\timesdots X_k \to X\mid f \in \varsign\}\), \(\syn{-} : \catset^{\varsorts^*\times \varsorts} \to \cattheories_\varsorts\) is left adjoint to the forgetful functor \(U_{\cattheories_\varsorts}\).
So, we call \(\syn{\varsign}\) the \emph{free} \(\varsorts\)-sorted Lawvere theory generated by \(\varsign\).

In the rest of this section, we shall see a correspondence between terms and morphisms in \(\syn{\varsign}\).
\begin{definition}
Let \(\varlaw\) be an \(\varsorts\)-sorted Lawvere theory.
\begin{itemize}
	\item Let \(\vec\pi: \vec X \to \vec Y\) in \(\varlaw\)  be a morphism from \(\vec X = X_1\timesdots X_n\) to \(\vec Y = Y_1\timesdots Y_m\) with \(m \le n\).
	We call \(\vec\pi\) a \emph{partial permutation} if there exists an injection \(u : \{1,\dots,m\} \to \{1,\dots,n\}\) such that \(X_{u(i)} = Y_i\) and \(\pi_i^{\vec Y}  \vec\pi = \pi_{u(i)}^{\vec X}\) for any \(i=1,\dots,m\).
A partial permutation is a \emph{permutation} if \(m=n\), i.e., \(u\) is a bijection.
	\item An \emph{efficient} morphism is a morphism \(\vec\omega : \vec X \to \vec Z\) in \(\varlaw\) satisfying the following: for any pair of a partial permutation \(\vec \pi : \vec X \to \vec Y\) and a morphism \(\vec\omega' : \vec Y \to \vec Z\), if \(\vec\omega = \vec\omega'\vec\pi\), then \(\vec\pi\) is a permutation.
\end{itemize}
\end{definition}
In \(\syn{\varsign,E}\), any partial permutation \(\vec\pi\) can be written as \((x_1,\dots,x_n \mid_E (x_{u(1)},\dots,x_{u(m)}))\) for some injection \(u : \{1,\dots,m\} \to \{1,\dots,n\}\), and \(u\) is a bijection iff \(\vec\pi\) is a permutation.
Also, any permutation is an isomorphism.

\begin{lemma}\label{lem:term_effmor}
For sorts \(X_1,\dots,X_m\in \varsorts\), 
let \(H = \{\vec\omega \in \Mor(\syn{\varsign}) \mid \text{\(\vec\omega\) is efficient and its codomain is } X_1\timesdots X_m\}\).
We define an equivalence relation \(\sim_\Pi\) on \(H\) as follows:
\(\vec\omega \sim_\Pi \vec\omega'\) iff there exists a permutation \(\vec\pi\) such that \(\vec\omega = \vec\omega' \vec\pi\).
Then, the following map is a well-defined bijection:
\begin{align*}
\phi : H/{\sim_\Pi} &\xrightarrow{\cong}\Term_\varsign^{X_1}(V)\timesdots \Term_\varsign^{X_m}(V)/{\sim_\alpha},\\
[\vec\omega]_{\sim_\Pi} &\mapsto [(t_1,\dots,t_m)]_{\sim_\alpha}\quad (\vec\omega = (\Gamma \mid_\emptyset (t_1,\dots,t_m))).
\end{align*}
Moreover, for any \([\vec\omega]_{\sim_\Pi} \in H_{\sim_\Pi}\) with \(\phi([\vec\omega]_{\sim_\Pi}) = [(t_1,\dots,t_m)]_{\sim_\alpha}\), \(\vec\omega\) is a permutation iff \((t_1,\dots,t_m)\) is an \(m\)-tuple of distinct variables.
\end{lemma}
\begin{proof}
We first show that \(\phi\) is well-defined.
Let \(\vec\omega = (\Gamma \mid_\emptyset (t_1,\dots,t_m)) \in H\) and \(\Gamma = y_1,\dots,y_n\).
For a permutation \(\vec\pi = (x_1,\dots,x_n \mid_\emptyset (x_{u(1)},\dots,x_{u(m)}))\), \(\vec\omega\vec\pi=(x_1,\dots,x_n \mid_\emptyset (t_1',\dots,t_m'))\) where \(t_i' = t_i[x_{u(1)}/y_1,\dots,x_{u(n)}/y_n]\).
Then, \((t_1,\dots,t_m)\sim_\alpha (t_1',\dots,t_m')\) by definition.

We construct an inverse \(\psi\) of \(\phi\).
Given terms \(t_1,\dots,t_m\) of sorts \(X_1,\dots,X_m\) respectively, let \(\{x_1,\dots,x_n\} = \Var(t_1)\cup\dots \cup \Var(t_m)\) and \(\vec\omega = (x_1,\dots,x_n \mid_\emptyset (t_1,\dots,t_m))\).
To show that \(\vec\omega\) is efficient, assume that \(\vec\omega = \vec\omega' \vec\pi\) for a morphism \(\vec\omega'\) and a partial permutation \(\vec\pi\).
Since \(\vec\pi\) can be written as \(\vec\pi = (x_1,\dots,x_n \mid_\emptyset (x_{u(1)},\dots,x_{u(n')}))\) for some injection \(u : \{1,\dots,n'\} \to \{1,\dots,n\}\) and \(\vec\omega' = (y_1,\dots,y_{n'} \mid_\emptyset (t_1',\dots,t_m'))\) for some terms \(t_1',\dots,t_m'\),
we have \(t_i = t_i'[x_{u(1)}/y_1,\dots,x_{u(n')}/y_{n'}]\) for any \(i=1,\dots,m\).
Since \(\{x_1,\dots,x_n\} = \Var(t_1)\cup\dots \cup \Var(t_m)\), \(u\) is a bijection, so \(\vec\pi\) is a permutation.
Therefore, \(\vec\omega\) is efficient.

Let \(\psi([(t_1,\dots,t_m)]_{\sim_\alpha}) = [\vec\omega]_{\sim_\Pi}\).
By the definition of \((\Gamma \mid_E (t_1,\dots,t_m))\),
we can see that \(\psi\) is well-defined.
Also, it is straightforward to check that \(\phi\) and \(\psi\) are inverses of each other.

The second statement of the lemma is obvious from the construction of \(\phi\).
\end{proof}

\section{Ringoids, modules, and algebraic discrete Morse theory}
In this section, we review the basic definitions and facts on ringoids and modules over ringoids which are developed in \cite{m72}.
Then, we introduce algebraic discrete Morse theory for chain complexes of modules over ringoids.
\subsection{Preliminaries on ringoids and modules}

A \emph{ringoid} is a small \(\catab\)-enriched category.
A \emph{two-sided ideal} of a ringoid \(\cR\) is a subobject of the Hom functor \(\cR(-,-) = \Hom_{\cR}(-,-) : \cR^\Op \times \cR \to \catab\) in the category of additive functors from \(\cR^\Op \times \cR\) to \(\catab\).
For a set \(U \subset \Mor(\cR)\), we write \((U)\) for the two-sided ideal
\[
(U)(X,Y) = \ab\{\sum_{r_1, r_2} r_1 u r_2 ~\middle|~ u \in U \cap \cR(i',j'), r_1 \in \cR(j',j), r_2 \in \cR(i,i')\}
\]
where the sum runs over finite combinations of \(r_1,r_2\).
For a ringoid \(\cR\) and a two-sided ideal \(I\) of \(\cR\), the \emph{quotient ringoid} is the ringoid \(\cR/I\) such that \(\Ob(\cR/I) = \Ob(\cR)\) and \(({\cR/I})(i,j) = \cR(i,j)/I(i,j)\) for any \(i,j \in \Ob(\cR)\).

For a quiver (that is, a directed multigraph) \(Q = (Q_0,Q_1)\), the \emph{ringoid freely generated} by \(Q\) is the ringoid \(\bbZ Q\) such that
\(\Ob(\bbZ Q) = Q_0\), and for any \(v_1,v_2 \in Q_0\), \({\bbZ Q}(v_1,v_2)\) is the free abelian group generated by the set of paths from \(v_1\) to \(v_2\) in \(Q\).
Also, for a set \(U \subset \Mor(\bbZ Q)\), the \emph{ringoid presented} by \((Q,U)\) is the quotient ringoid \(\bbZ Q/(U)\).

A \emph{(left) \(\cR\)-module} for a ringoid \(\cR\) is an additive functor \(\cR \to \catab\).
A \emph{right \(\cR\)-module}  is an additive functor \(\cR^\Op \to \catab\).
We often just say \(\cR\)-module for a left \(\cR\)-module in this paper.
An \(\cR\)-linear map between (either left or right) \(\cR\)-modules \(M,N\) is a natural transformation \(M \to N\).
\begin{proposition}\cite{m72}
The category of (either left or right) \(\cR\)-modules is an abelian category with enough projectives.
\end{proposition}

\begin{definition}
Let $\mathcal{R}$ be a ringoid and let $\cX = (\cX_i)_{i\in\Ob(\mathcal{R})}$ be a
family of sets. The \emph{free left $\mathcal{R}$–module on $\cX$} is the
left $\cR$-module $F(\cX)$ defined by
\[
  F(\cX)(i) \;=\; \bigoplus_{j\in\Ob(\mathcal{R})} \;
               \bigoplus_{x\in \cX_j} \mathcal{R}(i,j)\,{x},
\]
the direct sum of copies of $\mathcal{R}(i,j)$ indexed by generators
$x\in \cX_j$. A morphism of left $\cR$-modules $F(\cX)\to M$ is determined uniquely by
the images of the elements ${x}\in F(\cX)(j)$ for $x\in \cX_j$.
\end{definition}

\begin{proposition}
For any left $\cR$-module $M$, there is an exact sequence of \(\cR\)-modules
\[
  F(\cY) \xrightarrow{\;\phi\;} F(\cX) \longrightarrow M \longrightarrow 0
\]
where:
\begin{itemize}
  \item $\cX = (\cX_i)$ is a family of sets of \emph{generators} $x\in \cX_i$
    for $M(i)$;
  \item $\cY = (\cY_i)$ is a family of sets of \emph{relations}, where each
    $r\in \cY_i$ is a finite $\mathbb{Z}$-linear combination of elements of
    the form $u\cdot x$ with $u\in\mathcal{R}(i,j)$, $x\in \cX_j$;
  \item the map $\phi$ encodes these relations, and $M$ is, up to isomorphism,
    the unique module generated by the $\cX_i$ subject to the relations $\cY_i$.
\end{itemize}
\end{proposition}

Concretely, one often writes a presentation of a left $\mathcal{R}$-module $M$
by listing:
\begin{itemize}
  \item the generators $x\in \cX_j$, each labelled by the object $j$ at which it
    lives;
  \item the relations, which are formal $\mathbb{Z}$-linear combinations
    \[
      \sum_{k} u_k\cdot x_k = 0
    \]
    with $u_k\in \mathcal{R}(i,j_k)$ and $x_k\in \cX_{j_k}$.
\end{itemize}

\begin{definition}
The \emph{tensor product} of a right $\mathcal{R}$-module $M$ and a left
$\mathcal{R}$-module $N$ is the coend
\[
  M \otimes_{\mathcal{R}} N
  \;:=\;
  \int^{i\in\mathcal{R}} M(i) \otimes_{\mathbb{Z}} N(i),
\]
that is, the cokernel of the pair of maps
\[
  \bigoplus_{r\colon i\to j} 
    M(j)\otimes_{\mathbb{Z}} N(i)
  \rightrightarrows
  \bigoplus_{i\in\Ob(\mathcal{R})}
    M(i)\otimes_{\mathbb{Z}} N(i),
\]
where the two arrows are given on a generator
$m\otimes n\in M(j)\otimes N(i)$ by
\[
  M(r)(m)\otimes n
  \qquad\text{and}\qquad
  m\otimes N(r)(n),
\]
for each morphism $r\colon i\to j$ in $\mathcal{R}$.
\end{definition}

\subsection{Algebraic discrete Morse theory for ringoids}
This subsection is based on \cite{skoldberg06} but extended to modules over ringoids.

Let \(\cR\) be a ringoid and \(\bfK = (K_\bullet,\delta_\bullet)\) be a chain complex of left \(\cR\)-modules.
Suppose that each \(K_n\) is a free \(\cR\)-module generated by a family of sets \(I_n = \{I_{n,x}\}_{x\in \Ob(\cR)}\), and also suppose \(I_{n,x} \cap I_{m,y} = \emptyset \) unless \(n=m\), \(x=y\).
Then, each \(\delta_n : K_n \to K_{n-1}\) can be written in the form
\[
\delta_n(\alpha) = \sum_{\beta \in \bigcup_x I_{n-1,x}} [\alpha:\beta]_\delta{\beta}
\]
where \([\alpha:\beta]_\delta\) is a morphism from \(x\) to \(y\) in \(\cR\) for each \(\alpha \in I_{n,y}\), \(\beta \in I_{n-1,x}\).
We ommit the subscript \(\delta\) when there is no confusion.

Let \(G_\bfK = (V_\bfK, E_\bfK)\) be the directed (simple) graph whose set of vertices \(V_\bfK\) is the set \(\bigcup_{n,x} I_{n,x}\) and the set of edges \(E_\bfK\) consists of  \(\alpha \to \beta\) such that \([\alpha:\beta]\) is defined and nonzero.

\begin{definition}
A \emph{partial matching} \(\cM\) is a subset of \(E_\bfK\) such that no vertex is incident to more than one edge in \(\cM\).
For a partial matching \(\cM\), define \(G_\bfK^\cM = (V_\bfK^\cM, E_\bfK^\cM)\) as \(V_\bfK^\cM = V_\bfK\), \(E_\bfK^\cM = (E_\bfK \setminus \cM) \cup \{(\beta \to \alpha) \mid (\alpha \to \beta) \in \cM\}\).
\end{definition}

For a partial matching \(\cM\), a vertex \(\alpha\) is \(\cM\)-\emph{critical} or just \emph{critical} if there is no edge in \(\cM\) incident to \(\alpha\).
We write \(\Cr_n(\cM)\) for the set of critical vertices.
Also, a vertex \(\alpha\) is \(\cM\)-\emph{collapsible} or just \emph{collapsible} if there exists \(\beta\) such that \(\alpha \to \beta \in \cM\), and \(\alpha\) is \(\cM\)-\emph{redundant} or just \emph{redundant} if there exists \(\beta\) such that \(\beta \to \alpha \in \cM\).

We write \(\alpha \xrightarrow{+}_\cM \beta\) for the edge \(\beta \to \alpha\) in \(\cM\) and \(\alpha \xrightarrow{-}_\cM \beta\) for the edge \(\alpha \to \beta\) in \(E_\bfK \setminus \cM\).
Also, write \(\alpha \xrightarrow{+-}_\cM \beta\) (resp. \(\alpha \xrightarrow{-+}_\cM \beta\)) for a path \(\alpha \xrightarrow{+}_\cM \gamma \xrightarrow{-}_\cM \beta\) (resp. \(\alpha \xrightarrow{-}_\cM \gamma \xrightarrow{+}_\cM \beta\)) for some \(\gamma\).
%For each \(n\), let \(\succ_{\cM,n}\) be the partial order on \(I_n\) such that, for \(\alpha,\gamma \in I_n\), \(\alpha \succ_{\cM,n} \gamma\) if and only if there exists a path \(\alpha \to \beta \to \gamma\) in \(G_\bfK^\cM\).
%We often write just \(\succ_\cM\) instead of \(\succ_{\cM,n}\).

\begin{definition}
A partial matching \(\cM\) with the following conditions is called a \emph{Morse matching}:
(i) for any \(\alpha \to \beta \in \cM\), if \(\alpha \in I_{n,x}\), then \(\beta \in I_{n-1,x}\) and \([\alpha:\beta] = \pm 1_x\), and
(ii) there is no infinite path \(\alpha_1 \xrightarrow{+-}_{\cM} \alpha_2 \xrightarrow{+-}_{\cM} \dots\) in \(G_{\bfK}^\cM\).
\end{definition}

For \(\alpha \in \bigcup_x I_{n,x}\), \(\alpha' \in \bigcup_x I_{n-1,x}\), let \(\cM(\alpha,\alpha')\) be the set of paths \(p = \alpha = \alpha_0 \xrightarrow{-}_\cM \alpha_1 \xrightarrow{+}_\cM \dots \xrightarrow{-}_\cM \alpha_{2k+1} = \alpha'\) in \(G_\bfK^\cM\),
i.e., \(\alpha_0 \xrightarrow{-+}_\cM \alpha_2 \xrightarrow{-+}_\cM \alpha_4 \xrightarrow{-+}_\cM \dots \xrightarrow{-+}_\cM \alpha_{2k} \xrightarrow{-}_\cM \alpha_{2k+1}\).

\begin{theorem}
For a Morse matching $\cM$, let \(K_n^\cM = \bigoplus_{\alpha \in \Cr_n(\cM)} K_{n,\alpha}\).
Define \(\delta_n^\cM\) as
\begin{align*}
\delta_n^\cM &: K_n^\cM \to K_{n-1}^\cM\\
\delta_n^\cM(\alpha) &= \sum_{\alpha' \in \Cr_{n-1}(\cM)}\sum_{p \in \cM(\alpha,\alpha')}w(p),\\
w(p) &= (-1)^{k} 
[\alpha_{2k}:\alpha_{2k+1}][\alpha_{2k}:\alpha_{2k-1}]\dots[\alpha_2:\alpha_3][\alpha_2:\alpha_1][\alpha_0:\alpha_1] \alpha'
\\
&(p \text{ is the path }\alpha = \alpha_0 \to \dots \to \alpha_{2k+1} = \alpha' \text{ in } \cM(\alpha,\alpha')).
\end{align*}
Then, the sum in \(\delta_n^\cM\) is finite and \(\bfK^\cM = (K_\bullet^\cM, \delta_\bullet^\cM)\) is a chain complex that is homotopy equivalent to \(\bfK\).
\end{theorem}
The proof can be done in the same way as in \cite[Theorem 1]{skoldberg06}.

\section{Bar resolution for Lawvere theories}\label{sec:bar_resolution}

In the spirit of Beck \cite{beck67}, Barr \cite{barr02}, and Quillen \cite{quillen1970co,quillen2006homotopical}, a \emph{module} over an object \(C\) in an algebraic category \(\bfC\) is defined as an abelian group object in the slice category \(\bfC/C\).
Jibladze and Pirashvili \cite{jp06} showed that, for any \(\varsorts\)-sorted Lawvere theory \(\varlaw\), the category of modules over \(\varlaw\) is equivalent to the category of left modules over a certain ringoid \(\cU_\varlaw\) called the \emph{enveloping ringoid} of \(\varlaw\).
We give a construction of \(\cU_\varlaw\) below.

Let \(\varobj \in \Ob(\varlaw)\).
Define the ringoid \(\cR_\varlaw^{\varobj}\) by the following presentation \(\ab(Q^{\vec X}, U^{\vec X})\).
The set of vertices of \(Q^{\vec X}\) is \(\bigcup_{Y \in \varsorts}\varlaw(\vec X, Y)\).
The quiver \(Q^{\vec X}\) has an edge denoted by \(\partial_i(\omega)_{\vec\sigma}\) from \(\sigma_i\) to \(\omega\vec\sigma\) for each \(\vec\sigma = \ab<\sigma_1,\dots,\sigma_n> : \vec X \to \vec Y\), \(\omega : \vec Y \to Z\), and \(i=1,\dots,n\).
The set of relations \(U^{\vec X}\) contains
\[
\partial_i (\pi_i)_{\vec\sigma} = 1_{\sigma_i},\quad
\partial_i (\pi_j)_{\vec\sigma} = 0 \quad (i \neq j)
\]
for each $\vec\sigma = \ab<\sigma_1,\dots,\sigma_n>$
and 
\[
\partial_i(\omega\circ \ab<\omega_1',\dots,\omega_k'>)_{\vec\sigma}
- \sum_{j=1}^k \partial_j(\omega)_{\ab<\omega_1'\circ \vec\sigma,\dots,\omega_k'\circ \vec\sigma>} \partial_i(\omega_j')_{\vec\sigma}
\]
for each \(\omega : \varobj[1]'' = \varsort[1]''_1 \timesdots\varsort[1]''_k \to \varsort[2]\), \(\omega_j' : \varobj[1]' \to \varobj[1]''\) (\(j=1,\dots,k\)), \(\vec\sigma : \varobj[1] \to \varobj[1]'\).

For \(\vec\alpha : \varobj \to \varobj'\), define an additive functor \(\cR_\varlaw^{\vec\alpha} : \cR_\varlaw^{\varobj'} \to \cR_\varlaw^{\varobj}\) by
\begin{align*}
\cR_\varlaw^{\vec\alpha}(\vec\sigma) &= \vec\sigma \vec\alpha\quad \text{for an object $\vec \sigma$ in $\cR^{\vec X'}_\varlaw$},\\
\cR_\varlaw^{\vec\alpha} \partial_i(\omega)_{\vec\sigma} &= \partial_i(\omega)_{\vec\sigma \vec\alpha} \quad \text{for a morphism $\partial_i(\omega)_{\vec\sigma}$ in $\cR^{\vec X'}_\varlaw$}.
\end{align*}

\begin{definition}
The enveloping ringoid \(\cU_\varlaw\) of \(\varlaw\) is defined as follows.
\begin{itemize}
	\item The objects of \(\cU_\varlaw\) are the morphisms in \(\varlaw\) whose codomains are in \(\varsorts\).
	\item For \(f : \varobj' \to \varsort[1]\), \(g : \varobj \to \varsort[2]\) in \(\varlaw\),
\[
{\cU_\varlaw}(f, g) = \bigoplus_{\vec\alpha : \varobj \to \varobj'} {\cR_\varlaw^{\varobj}}(f\circ \vec\alpha,g)
\]
	\item The composition is given as follows. For \(r : f \to g\), $r' : g \to h$, (\(f : \varobj \to \varsort[1]\), \(g : \varobj' \to \varsort[1]'\), \(h : \varobj'' \to \varsort[1]''\)), the $\vec\beta$-th component of \(r' \circ r\) (\(\vec\beta : \varobj'' \to \varobj\)) is given by
\[
(r'\circ r)_{\vec\beta} = \sum_{\vec\alpha : \varobj' \to \varobj, \vec\alpha' : \varobj'' \to \varobj', \vec\alpha \circ \vec\alpha' = \vec\beta} r'_{\vec\alpha'} \circ \cR_\varlaw^{\vec\alpha'}r_{\vec\alpha}.
\]
\end{itemize}
\end{definition}

\begin{remark}
In \cite[3.3]{jp06}, the enveloping ringoid of \(\varlaw\) is called the \emph{total ringoid} \(\cR_\varlaw[\varlaw^\Op]\) of the ringoid-valued functor \(\varobj \mapsto \cR_\varlaw^{\varobj}\).
\end{remark}

We write \(\partial_{i_1}(\omega_1)_{\vec\sigma_1}\dots \partial_{i_k}(\omega_k)_{\vec\sigma_k}\) instead of \(\In_{(1_{\varobj})}(\partial_{i_1}(\omega_1)_{\vec\sigma_1}\dots \partial_{i_k}(\omega_k)_{\vec\sigma_k})\).
Also, we write \(\vec\alpha^*\) instead of \(\In_{\vec\alpha}(1_{\omega\vec\alpha})\) for any \(\omega\).
Then, the following is immediate by definition.
\begin{lemma}\label{eqn:alpha_star}
\begin{enumerate}
	\item \(\In_{\vec\alpha}(\partial_1(\omega_1)_{\vec\sigma_1}\dots\partial_{i_k}(\omega_k)_{\vec\sigma_k}) = \partial_{i_1}(\omega_1)_{\vec\sigma_1}\dots \partial_{i_k}(\omega_k)_{\vec\sigma_k}\vec\alpha^* \),
	\item \(\vec\alpha^* \partial_i(\omega)_{\vec\sigma} = \In_{\vec\alpha}(\partial_i(\omega)_{\vec\sigma \vec\alpha}) = \partial_i(\omega)_{\vec\sigma\vec\alpha}\vec\alpha^*\).
\end{enumerate}
\end{lemma}
By this lemma, we can see that any morphism of \(\cU_\varlaw\) can be written as a finite sum of morphisms of the form \(\partial_{i_1}(\omega_1)_{\vec\sigma_1}\dots \partial_{i_k}(\omega_k)_{\vec\sigma_k}\vec\alpha^*\).

\begin{definition}
The \(\cU_\varlaw\)-\emph{module of K\"ahler differentials} \(\kahlerdiffs_\varlaw\) is defined by generators and relations as follows.
For each \(\omega : \varsort[1]_1 \timesdots \varsort[1]_n \to \varsort[2]\), \(\kahlerdiffs_\varlaw\) has a generator \(d(\omega)\) in \(\kahlerdiffs_\varlaw(\omega)\), and relations
\[
d(\omega \circ \ab<\omega_1',\dots,\omega_n'>) = \sum_{i=1}^n \partial_i(\omega)_{\ab<\omega_1',\dots,\omega_n'>} d(\omega_i'),
\]
\[
\vec\alpha^*d(\omega) = d(\omega\vec\alpha).
\]
\end{definition}
Note that, for any $\omega : \vec X = \varsort_1\timesdots \varsort_n \to \varsort$, since \(\omega = \omega 1_{\vec X} = \omega\ab<\pi_1,\dots,\pi_n>\), we have \(d(\omega) = \sum_{i=1}^n \partial_i(\omega)_{1_{\vec X}} d(\pi_i) = \sum_{i=1}^n \partial_i(\omega)_{1_{\vec X}} \pi_i^* d1_{\varsort_i}\) in \(\kahlerdiffs_\varlaw\).

\subsection{Unnormalized bar resolution}\label{subsec:unnormalized}

In this subsection, we define the unnormalized bar resolution of \(\kahlerdiffs_\varlaw\), \(\dots \xrightarrow{\delta_2^u} B_1^u \xrightarrow{\delta_1^u} B_0^u \xrightarrow{\delta_0^u} \kahlerdiffs_\varlaw\).

For \(n > 0\), define \(B_n^u\) as the free \(\cU_\varlaw\)-module generated by
\[
(\omega_1, \vec\omega_2, \dots, \vec\omega_n) \in B_n^u(\omega_1\vec\omega_2  \dots  \vec\omega_n)
\]
for each \(n\)-tuple of \(\omega_1 : \varobj_1 \to \varsort[1]\), \(\vec\omega_i : \varobj_i \to \varobj_{i-1}\)
(\(\vec X_i \in S^*, Y \in S\), \(i=2,\dots,n\)).
For \(n=0\), define \(B_0^u\) as the free \(\cU_\varlaw\)-module generated by
\[
()_{\varsort} \in B_0^u\ab(1_{\varsort})
\]
for each \(\varsort \in \varsorts\).
First, define \(\delta^u_0 : B_0^u \to \kahlerdiffs_\varlaw\) and \(\delta_1^u : B_1^u \to B_{0}^u\) as
\begin{align*}
\delta_0^u\ab(()_{\varsort}) &= d\ab(1_{\varsort}), \quad (X \in \varsorts)\\
\delta_1^u(\omega) &= \sum_{i=1}^n \partial_i(\omega)_{1_{\varobj}}\ab(\pi_i^{\varobj})^* ()_{\varsort_i} - \omega^*()_{\varsort[1]}
\quad (\omega : \vec X \to Y, \vec X \in \varsorts^*, Y \in \varsorts).
\end{align*}

To define \(\delta_n^u : B_n^u \to B_{n-1}^u\) for \(n>1\), we give a semi-simplicial structure \(d_j : B_n^u \to B_{n-1}^u\) (\(j = 0,\dots,n\)) as
\begin{align*}
d_0(\omega_1, \vec\omega_2, \dots, \vec\omega_n) &=
\sum_{i=1}^{m_2} \partial_i(\omega_1)_{\vec \omega_2} (\omega_{2,i},\vec\omega_3, \vec\omega_4,\dots,\vec\omega_n) \quad (\vec \omega_2 = \ab<\omega_{2,1},\dots,\omega_{2,m_2}>),\\
d_j(\omega_1, \vec\omega_2, \dots, \vec\omega_n) &=
(\omega_1,\vec\omega_2,\dots,\vec\omega_{j-1},\vec\omega_j\vec\omega_{j+1}, \vec\omega_{j+2},\dots,\vec \omega_n) \quad (0 < j < n),\\
d_n(\omega_1, \vec\omega_2, \dots, \vec\omega_n) &= \vec\omega_n^*(\omega_1, \vec\omega_2, \dots, \vec\omega_{n-1}).
\end{align*}

Then, define \(\delta_n^u : B_n^u \to B_{n-1}^u\) as
\[
\delta_{n}^u(\omega_1, \vec\omega_2, \dots, \vec\omega_n) =
\sum_{k=0}^n (-1)^k d_k(\omega_1, \vec\omega_2, \dots, \vec\omega_n).
\]

We can show \(\delta_n^u\delta_{n+1}^u = 0\) by the following lemma.
\begin{lemma}
\(d_id_j = d_{j-1}d_i\) if \(i < j\).
\end{lemma}
\begin{proof}
Suppose \(i=0\) and \(j = 1\). Then, we have
\begin{align*}
d_0d_1(\omega_1, \vec\omega_2, \dots, \vec\omega_n) &=
d_0(\omega_1\vec\omega_{2}, \vec\omega_3, \dots, \vec\omega_n) \\
&= \sum_{j=1}^{m_2} \partial_j(\omega_1\vec\omega_2) (\omega_{3,j},\vec\omega_4,\dots,\vec\omega_n) \quad (\vec \omega_3 = \ab<\omega_{3,1},\dots,\omega_{3,m_3}>)\\
d_0d_0(\omega_1, \vec\omega_2, \dots, \vec\omega_n) &=
\sum_{i=1}^{m_2} \partial_i(\omega_1) d_0'(\omega_{2,i},\vec\omega_3,\dots,\vec\omega_n) \quad (\vec \omega_2 = \ab<\omega_{2,1},\dots,\omega_{2,m_2}>)\\
&= \sum_{i=1}^{m_2} \partial_i(\omega_1) \sum_{j=1}^{m_3} \partial_k(\omega_{2,i}) (\omega_{3,j},\vec\omega_4,\dots,\vec\omega_n).
\end{align*}
The other cases are straightforward.
\end{proof}

To show that \(\ker \delta_{n+1}^u = \im \delta_{n}^u\),
we define abelian group homomorphisms \(s_0 : \kahlerdiffs_\varlaw(\omega) \to B_0^u(\omega)\) and \(s_{n+1} : B_n^u(\omega) \to B_{n+1}^u(\omega)\) for each \(\omega\) in \(\varlaw\) and \(n \ge 0\) that satisfies
\begin{equation}\label{eqn:cont_hom}
s_{n+1}(\partial_i(\omega')_{\vec\sigma} c) = \partial_i(\omega')_{\vec\sigma} s_{n+1}(c)
\end{equation}
for any \(c \in B_n^u(\omega)\).
For \(d\omega  = \sum_{i=1}^n \partial_i(\omega)_{\ab<\pi_1,\dots,\pi_n>}\pi_i^*d1_{X_i} \in \kahlerdiffs_\varlaw(\omega)\), let \(s_0(d\omega) = \sum_{i=1}^n\partial_i(\omega)_{\ab<\pi_1,\dots,\pi_n>}\pi_i^*()_{X_i}\).
It is easy to see that \(s_0\) is well-defined.

For \(c=\vec\alpha^*(\omega_1,\dots,\vec\omega_n) \in B_n^u(\omega_1\dots\vec\omega_n\vec\alpha)\), let \(s_{n+1}(c) = (-1)^n(\omega_1,\dots,\vec\omega_n,\vec\alpha)\) and then \(s_{n+1}\) extends to a unique abelian group homomorphism satisfying \eqref{eqn:cont_hom}.

\begin{lemma}
\(\delta_0s_0 = 1_{B_0^u}\), and, for each \(n\), \(s_n\delta_n + \delta_{n+1}s_{n+1} = 1_{B_n^u}\).
\end{lemma}
\begin{proof}
The first equality is shown as
\begin{align*}
\delta_0\ab(s_0(d\omega)) &= \delta_0\ab(\sum_{i=1}^n \partial_i(\omega)_{\ab<\pi_1,\dots,\pi_n>} \pi_i^* ()_{\varsort_i})  = \sum_{i=1}^n \partial_i(\omega)_{\ab<\pi_1,\dots,\pi_n>} \pi_i^* d(1_{\varsort_i})  = d\omega.
\end{align*}

For the second equality with \(n=0\), we have
\begin{align*}
s_0\ab(\delta_0\ab(\alpha^* ()_{\varsort[1]})) &= s_0\ab(d(\alpha) )\\
&= s_0\ab( \sum_{j=1}^m \partial_j(\alpha)_{\ab<\pi_1,\dots,\pi_m>} \pi_j^* d(1_{\varsort_j}) )\\
&= \sum_{j=1}^m \partial_j(\alpha)_{\ab<\pi_1,\dots,\pi_m>} (\pi_j^{\varobj})^* ()_{\varsort_j}
\end{align*}
\begin{align*}
\delta_1\ab(s_1(\alpha^* ()_{\varsort[1]})) &= -\delta_1\ab( (\alpha) ) \\
&= -\sum_{j=1}^m \partial_j(\alpha)_{\ab<\pi_1,\dots,\pi_m>} (\pi_j^{\varobj})^* ()_{\varsort_j} +  \alpha^* ()_{\varsort[1]}
\end{align*}
Thus, \(s_0\delta_0+ \delta_1 s_0 = 1_{B_0^u}\).

To prove \(s_n\delta_n + \delta_{n+1}s_{n+1} = 1_{B_n^u}\) for \(n > 0\),
it is enough to show \(d_{n+1}s_{n+1} = 1_{B_{n+1}^u}\) and \(d_ks_{n+1} = - s_n d_k\) for \(k = 0,\dots,n\)
because then we have
\begin{align*}
\delta_{n+1}s_{n+1} = \sum_{k=0}^{n+1} (-1)^k d_k s_{n+1} &= \sum_{k=0}^{n} (-1)^k s_n d_k + (-1)^{2(n+1)}1_{B_n^u} \\
&= - s_n \delta_n + 1_{B^u_n}.
\end{align*}
We can show \(d_{n+1}s_{n+1} = 1_{B_{n+1}^u}\) as
\begin{align*}
d_{n+1} s_{n+1}(\vec\alpha^*(\omega_1,\vec\omega_2,\dots,\vec\omega_n))
&= d_{n+1}\ab((-1)^{n+1} (\omega_1,\vec\omega_2,\dots,\vec\omega_n,\vec\alpha) ) \\
&= (-1)^{n+1}\vec \alpha^*(\omega_1,\vec\omega_2,\dots,\vec\omega_n),
\end{align*}
and 
\(
d_k s_{n+1} = - s_n d_k
\)
is also straightforward.
\end{proof}

\begin{remark}
When we have \((\omega_1,\vec\omega_2,\dots,\vec\omega_n)\), we often write like ``\(\vec\omega_i\) (\(i=1,\dots,n\))'', and for \(i=1\) it means that \(\vec\omega_1=\omega_1\).
\end{remark}

\subsection{Normalized bar resolution}

In this subsection, we define normalized bar resolution by applying algebraic discrete Morse theory to the unnormalized bar resolution defined in the previous subsection.
Here is the informal idea:
We want to regard a cell \((\omega_1,\vec\omega_2,\dots,\vec\omega_n)\) as ``normalized'' if 
\(\vec\omega_i\) is not a partial permutation for each \(i\in\{1,\dots,n\}\).
In addition, for any permutation \(\vec\pi\), we do not want to regard both of the two cells \((\omega_1,\vec\omega_2,\dots,\vec\omega_n)\), \((\omega_1,\vec\omega_2,\dots,\vec\omega_i\vec\pi,\vec\pi^{-1}\vec\omega_{i+1},\dots,\vec\omega_n)\) as ``normalized''.
We shall define a Morse matching \(\cM'\) such that \(\cM'\)-critical cells are exactly the cells that are ``normalized'' in this sense.

We define an equivalence relation \(\sim_\Pi\) on morphisms in \(\varlaw\) as \(\vec\omega \sim_\Pi \vec\omega' \iff\vec \omega =\vec \omega' \vec\pi\) for some permutation \(\vec\pi\).
For each equivalence class \([\vec\omega]_{\sim_\Pi}\), we choose a representative \(\vec\omega' \in [\vec\omega]_{\sim_\Pi}\).
For the equivalence class of an identity, we choose the identity itself as the representative.
Note that any morphism \(\vec\omega\) can be uniquely decomposed as \(\vec\omega = \vec\omega' \vec\pi\) for the chosen representative \(\vec \omega'\) and some permutation \(\vec\pi\).

A morphism \(\vec\omega : \vec X \to \vec Y\) is \emph{essential} if
\(\vec\omega\) is efficient and the chosen representative of its equivalence class \([\vec\omega]_{\sim_\Pi}\).
\begin{lemma}
Any morphism \(\vec\omega\) can be uniquely decomposed as \(\vec\omega = \vec\omega' \vec\pi\) for some essential \(\vec\omega'\) and some partial permutation \(\vec\pi\).
\end{lemma}
\begin{proof}
There is at least one such decomposition since, if \(\vec\omega\) is essential, we can take \(\vec\omega' = \vec\omega\) and \(\vec\pi = \Id\), and if \(\vec\omega\) is not essential, the existence follows by definition.

Suppose that there are two such decompositions, \(\vec\omega = \vec\omega' \vec\pi' = \vec\omega'' \vec\pi''\) for essential \(\vec\omega',\omega''\) and partial permutations \(\vec\pi',\vec\pi''\).
Since \(\vec\pi', \vec\pi''\) are partial permutations, they can be written as \(\vec\pi' = \ab<\pi_{u(1)},\dots,\pi_{u(n)}>\), \(\vec\pi'' = \ab<\pi_{v(1)},\dots,\pi_{v(k)}>\) for some injections \(u : \{1,\dots,n\} \to \{1,\dots,m\} \), \(v : \{1,\dots,k\} \to \{1,\dots,m\}\).
Let \(w : \{1,\dots,m\} \to \{1,\dots,n\}\) be an arbitrary section of \(u\), i.e., a function satisfying \(wu(i) = i\) for any \(i =1,\dots,n\),
and let \(\vec\iota = \ab<\pi_{w(1)},\dots,\pi_{w(m)}>\).
Then, \(\vec\pi'\vec\iota\) is an identity, hence
\(\vec\omega\vec\iota = \vec\omega' = \vec\omega''\vec\pi''\vec\iota\).
Since \(\vec\omega'\) is essential, \(\vec\pi''\vec\iota = \ab<\pi_{wv(1)},\dots,\pi_{wv(k)}>\) must be an identity, so \(k = n\) and since \(w\) is an arbitrary section of \(u\), \(v = u\).
\end{proof}

For a cell \(\tau = (\omega_1,\vec \omega_2,\dots,\vec\omega_n)\) and an index \(i \in \{1,\dots,n-1\}\),
consider the following six predicates:
\begin{description}
	\item[\(P(\tau,i)\) :] \(i < n\), \(\vec\omega_i\) is essential, \(\vec\omega_{i+1}\) is a partial permutation, and \(\vec\omega_i\vec\omega_{i+1}\) is neither essential nor a partial permutation,
	\item[\(P'(\tau,i)\) :] \(\vec\omega_i\) is not essential and not a partial permutation,
	\item[\(Q(\tau,i)\) :] \(\vec\omega_i\) is an identity and, for the largest integer \(i' \in \{i,\dots,n\}\) such that all \(\vec\omega_i,\dots,\vec\omega_{i'}\) are identities, \(i'-i\) is odd,
	\item[\(Q'(\tau,i)\) :] the same as \(Q(\tau,i)\) but \(i' - i\) is even.
	\item[\(R(\tau,i)\) :] \(i=1\), \(\omega_1 = \pi_k : X_1\timesdots X_n \to X_k\) for some $n > 1$ and \(k < n\), \(\vec\omega_2\) is an identity, and \(j\) is odd for the smallest integer \(j > 1\) such that \(\vec\omega_j\) is not an identity.
	\item[\(R'(\tau,i)\) :] the same as \(R(\tau,i)\) but \(j\) is even.
\end{description}
Define a set of edges \(\cM'\) as
\[
\tau = (\omega_1,\vec\omega_2,\dots,\vec\omega_{i},\vec\omega_{i+1},\dots,\vec\omega_n) \to
(\omega_1,\vec\omega_2,\dots,\vec\omega_{i}\vec\omega_{i+1}, \dots,\vec\omega_n) \in \cM'
\]
if and only if either following (1) or (2) holds: (1) \(R(\tau,i)\) holds, (2) \(P(\tau,i)\) or \(Q(\tau,i)\) holds and \(\vec\omega_j\) is essential and not an identity for each \(j < i\).

\begin{lemma}
\(\cM'\) is a partial matching.
\end{lemma}
\begin{proof}
It is obvious by definition that there is no two edges in \(\cM'\) that have the same source.

Suppose that there are two edges with a common target, \(\tau = (\omega_1,\vec \omega_2,\dots,\vec\omega_i,\vec\omega_{i+1},\dots,\vec\omega_n) \to \tau''\) and \(\tau' = (\omega_1',\vec\omega_2',\dots,\vec\omega_i',\vec\omega_{i+1}',\dots,\vec\omega_n') \to \tau''\).
Then, we have \(\omega_1 = \omega_1'\), \(\vec\omega_j = \vec\omega_j'\) for \(j \neq i,i+1\), and \(\vec\omega_i \vec\omega_{i+1} = \vec\omega_i'\vec\omega_{i+1}'\).
\begin{itemize}
	\item Case where (1) for $(\tau,i)$ and (1) for $(\tau',i)$ hold:
		By $R(\tau,i)$, we have $i=1$, $\omega_1 = \pi_k$ for some $k$, and $\vec \omega_2 = \Id$, and by $R(\tau',i)$ and $\omega_1\vec \omega_2 = \omega_1'\vec\omega_2'$, we have $\omega_1 = \omega_1'$ and $\vec \omega_2 = \vec \omega_2'$.
	\item Case where (1) for $(\tau,i)$ and (2) for $(\tau',i)$ hold:
		We show that this case cannot happen. 
		As above, $i=1$, $\omega_1 = \pi_k$ for some $k$, and $\vec\omega_2 = \Id$.
		Assume that $P(\tau',1)$ holds. Then, $\omega_1' $ is essential and $\vec\omega_2'$ is a partial permutation. Since $\pi_k = \omega_1'\vec\omega_2'$ and $\pi_k = \Id \pi_k$, we have $\omega_1' = \Id$ by the previous lemma, and this contradicts to $P(\tau',1)$.
		Assume that $Q(\tau',1)$ holds. In this case, $\omega_1' = \vec\omega_2'=\Id$, but this contradicts to $\omega_1\vec\omega_2 = \pi_k : X_1\timesdots X_n \to X_k$ for $n > 0$.
	\item Case where (2) for $(\tau,i)$ and (2) for $(\tau',i)$ hold:
		Note that the only two subcases are possible: both $P(\tau,i)$ and $P(\tau',i)$ hold or both $Q(\tau,i)$ and $Q(\tau',i)$ hold.
		In the former subcase, we can apply the previous lemma, and the latter subcase, it is immediate that $\omega_i = \vec\omega_i' = \vec\omega_{i+1} = \vec\omega_{i+1}' = \Id$.
\end{itemize}
\end{proof}

By definition, \(\tau\) is \(\cM'\)-collapsible iff (1) or (2) holds.
We say that an \(\cM'\)-collapsible cell \(\tau\) is of \emph{type} \(R\) if \(\tau\) satisfies (1), and of \emph{type} \(P\) (resp. \(Q\)) if \(\tau\) satisfies (2) with \(P(\tau, i)\) (resp. \(Q(\tau,i)\)) for some \(i\).

\begin{lemma}
\(\tau' = (\omega_1,\vec\omega_2,\dots,\vec\omega_n)\) is \(\cM'\)-redundant if and only if either (1') \(R'(\tau',1)\) or (2') there exists \(i\) such that \(P'(\tau',i)\) or \(Q'(\tau',i)\) holds, and, for any \(j < i\), \(\vec\omega_j\) is essential and not an identity.
\end{lemma}
\begin{proof}
It is straight forward to check the following three equivalences.
First, \(R'(\tau',1)\) holds iff \(R(\tau,1)\) holds for \(\tau = (\omega_1,\Id,\vec\omega_2,\dots,\vec\omega_n)\).
Next, for any index \(i\), (2') holds with \(P'(\tau',i)\) iff 
(2) holds with \(P(\tau,i)\) for \(\tau = (\omega_1,\vec\omega_2,\dots,\vec\omega_{i-1}, \vec\omega_i',\vec\pi,\vec\omega_{i+1},\dots,\vec\omega_n)\) where \(\vec\omega_i'\) is essential, \(\vec\pi\) is a partial permutation, and \(\vec\omega_i = \vec\omega_i'\vec\pi\).
Finally, (2') holds with \(Q'(\tau',i)\) iff (2) holds with \(Q(\tau,i)\)
for \(\tau = (\omega_1,\vec\omega_2,\dots,\vec\omega_{i-1}, \Id, \vec\omega_{i+1},\dots,\vec\omega_n)\).
In either case, \(\tau \to \tau'\) is in \(\cM'\), so the ``if''-part follows.

Suppose that \(\tau'\) is \(\cM'\)-redundant for the ``only if''-part.
Then, there exist \(i\), \(\vec\omega_i',\vec\omega_i''\) such that \(\vec\omega_i = \vec\omega_i'\vec\omega_i''\) and \(\tau \to \tau'\) is in \(\cM'\) for \(\tau = (\omega_1,\vec\omega_2,\dots,\vec\omega_i',\vec\omega_i'',\dots,\vec\omega_n)\).
In particular, \(\tau\) is \(\cM'\)-collapsible, so satisfies (1) or (2).
If \(R(\tau,i)\) holds, then \(i=1\), \(\omega_1 = \omega_1' =\pi_k\), and \(\vec\omega_1'' = \Id\), so \(R(\tau',1)\) holds.
The case where (2) holds can be proved similarly.
\end{proof}

\begin{lemma}
\(\tau = (\omega_1,\vec\omega_2,\dots,\vec\omega_n)\) is \(\cM'\)-critical iff \(\vec\omega_i\) is essential and not an identity for any \(i\).
\end{lemma}
\begin{proof}
It is obvious that some of \(\omega_1,\dots,\vec\omega_n\) are not essential or identities if \(\tau\) is \(\cM'\)-collapsible or \(\cM'\)-redundant.
To show the other direction, suppose that some \(\vec\omega_i\) is not essential or an identity, and suppose \(i\) is the smallest among such indices.
If \(\vec\omega_i\) is an identity, \(Q(\tau,i)\) or \(Q'(\tau,i)\) holds, so \(\tau\) is not \(\cM'\)-critical.
If \(\vec\omega_i\) is a partial permutation and not an identity, since \(\vec\omega_{i-1}\) is essential and not an identity, \(P(\tau,i-1)\) holds, so \(\tau\) is \(\cM'\)-collapsible.
If \(\vec\omega_i\) is not essential and not a partial permutation, \(P'(\tau,i)\) holds so \(\tau\) is \(\cM'\)-redundant.
\end{proof}

For \(\tau = (\omega_1,\vec\omega_2,\dots,\vec\omega_n)\), let \(I(\tau)\) be the largest index \(i\) such that \(\omega_1,\vec\omega_2,\dots,\vec\omega_{i-1}\) are all essential and not identities.
Here, \(I(\tau)= 1\) means that \(\omega_1\) is not essential or is an identity.
Also, let \(N(\tau)\) be the number of \(i \in \{1,\dots,n\}\) such that \(\vec\omega_i\) is a partial permutation.

\begin{lemma}
For any path \(\tau \xrightarrow{+-}_{\cM'} \tau''\), we have \(N(\tau) \le N(\tau'')\).
\end{lemma}
\begin{proof}
For any edge \(\tau \xrightarrow{+}_{\cM'} \tau'\) with \(\tau = (\omega_1,\vec \omega_2,\dots,\vec\omega_i\vec\omega_{i+1},\dots,\vec\omega_n)\), \(\tau' = (\omega_1,\vec\omega_2,\dots,\vec\omega_n)\), by conditions \(P(\tau',i)\), \(Q(\tau',i)\), \(R(\tau',i)\), we can see that \(\vec\omega_{i+1}\) is a partial permutation, and either \(\vec\omega_i\vec\omega_{i+1}\) is not a partial permutation, \(\vec\omega_i = \Id\), or \(\vec\omega_i = \pi_k\) for some \(k\).
In either case, we have \(N(\tau') = N(\tau) + 1\).
Also, for any edge \(\tau' \xrightarrow{-}_{\cM'} \tau''\), \(N(\tau')-1 \le N(\tau'')\) obviously.
Therefore, for any path \(\tau \xrightarrow{+-}_{\cM'} \tau''\), we have \(N(\tau) \le N(\tau'')\).
\end{proof}

\begin{lemma}\label{lem:normalized_wf}
If \(\tau \xrightarrow{+-}_{\cM'} \tau''\) and \(\tau''\) is \(\cM'\)-redundant, then \(N(\tau) < N(\tau'')\) or \(I(\tau) < I(\tau'')\).
\end{lemma}
\begin{proof}
Consider an edge \(\tau' = (\omega_1,\vec\omega_2,\dots,\vec\omega_n) \to (\omega_1,\vec\omega_2,\dots,\vec\omega_i\vec\omega_{i+1},\dots,\vec\omega_n) = \tau\) in \(\cM\), i.e., \(\tau \xrightarrow{+}_{\cM'} \tau'\).
If \(\tau' \xrightarrow{-}_{\cM'} \tau''\), then \(\tau''\) can be written as either
(a) \((\omega_{2,k},\vec\omega_3,\dots,\vec\omega_n)\) for some \(k \in \{1,\dots,m\}\) where \(\vec\omega_2 = \ab<\omega_{2,1},\dots,\omega_{2,m}>\),
(b) \((\omega_1,\vec\omega_2,\dots,\vec\omega_{n-1})\), or
(c) \(d_j(\tau')\) for \(j = 1,\dots,n-1\), \(j\neq i\).

Case (a): We have \(N(\tau) < N(\tau'')\) if \(\omega_1\) is not a partial permutation.
Assume that \(\omega_1\) is a partial permutation and imply a contradiction.
Note that \(\omega_1 = \Id_X\) for \(X \in \varsorts\) or \(\omega_1 = \pi_k\) for some \(k\).
If \(\omega_1 = \Id_X\), then since \(\tau'\) is \(\cM'\)-collapsible, \(Q(\tau',1)\) holds, so \(\vec\omega_2 = \Id_X\) and \(\tau'' = \tau\).
If \(\omega_1 = \pi_k\), then since \(\tau'\) is \(\cM'\)-collapsible, \(R(\tau',1)\) holds, so \(\vec\omega_2 = \Id = \ab<\pi_1,\dots,\pi_m>\) and \(\tau'' = \tau\).
Therefore, in either case, the edge \(\tau' \to \tau''\) is in \(\cM'\) and this contradicts to \(\tau' \xrightarrow{-}_{\cM'} \tau''\).

Case (b): If \(\tau\) is of type \(P'\) and \(\tau'\) is of type \(P\),
then \(I(\tau'') = I(\tau') = I(\tau)+1\).
Suppose that \(\tau\) is of type \(Q'\) and \(\tau'\) is of type \(Q\).
If \(\vec\omega_i,\dots\,\vec\omega_{n}\) are all identities, then \(\tau'' = \tau\) and it contradicts to \(\tau' \xrightarrow{-}_{\cM'} \tau''\).
Otherwise, \(\tau''\) must be \(\cM'\)-collapsible of type \(Q\).

Case (c):
If \(j < i\), we have \(N(\tau) < N(\tau'')\) since \(\vec\omega_j,\vec\omega_{j+1}\) are not partial permutations.
If \(j > i\), we have \(I(\tau) < I(\tau'')\) by a similar argument with Case (b).
\end{proof}
\begin{theorem}
\(\cM'\) is a Morse matching.
\end{theorem}
\begin{proof}
From the previous two lemmas, we can conclude that there are no infinite sequences \(\tau_1 \xrightarrow{+-}_{\cM'} \tau_2 \xrightarrow{+-}_{\cM'} \dots\) of \(n\)-cells since the lexicographic pairs \((N(\tau_k),I(\tau_k))\) are increasing but, at the same time, at most \((n,n)\).

Also, it is easy to check that \([\tau:\tau']_{\delta^u} = \Id\) or \(-\Id\) for any edge \(\tau \to \tau'\) in \(\cM'\).
\end{proof}

Let \((B_\bullet,\delta_\bullet) = \ab(B_\bullet^{u,\cM'}, \delta_\bullet^{u,\cM'})\).
We shall consider how \(\delta_n(\tau)\) can be described explicitly.

Let \(\Phi : B_{n-1}^u \to B_{n-1}\) be the map defined as follows.
For \(\tau = (\omega_1,\vec\omega_2,\dots,\vec\omega_{n-1})\), if \(\tau\) is \(\cM'\)-critical, then \(\Phi(\tau) = \tau\).
If some \(\vec\omega_i\) is a partial permutation, then \(\Phi(\tau) = 0\).
Otherwise, decompose \(\vec\omega_{I(\tau)}\), which is not essential, as \(\vec\omega_{I(\tau)} = \vec\omega_{I(\tau)}' \vec\pi\) for some essential \(\vec\omega_{I(\tau)}'\) and some partial permutation \(\vec\pi\).
If \(I(\tau) < n-1\),
\(\Phi(\tau) = \Phi(\omega_1,\vec\omega_2,\dots,\vec\omega_{I(\tau)}',\vec\pi\vec\omega_{I(\tau)+1},\dots,\vec\omega_{n-1})\).
If \(I(\tau) = n-1\),
\(\Phi(\tau) = \vec\pi^* (\omega_1,\vec\omega_2,\dots,\vec\omega_{n-1}')\).
\begin{lemma}
\(\delta_n(\tau) = \Phi\delta_n^u(\tau)\).
\end{lemma}
\begin{proof}
For an \(\cM'\)-critical cell \(\tau'\), suppose that \(\cM'(\tau,\tau')\) contains a path \(p\),
\[
\tau \xrightarrow{-}_{\cM'} \tau^0 \xrightarrow{+-}_{\cM'} \tau^1 \xrightarrow{+-}_{\cM'} \dots \xrightarrow{+-}_{\cM'} \tau^k =\tau'
\]
where \(\tau^l = (\omega_1^l,\vec\omega_2^l,\dots,\vec\omega_{n-1}^l)\) is an \(\cM'\)-redundant cell for each \(l=0,\dots,k\).
Notice that \(\tau^l\) does not contain partial permutations, i.e., \(N(\tau^l) = 0\),
since \(\tau^{k}=\tau'\) is \(\cM'\)-critical, which implies \(N(\tau')= 0\), and \(N(\tau^l) \le N(\tau')\).
Therefore, each \(\tau^l\) is not of type \(Q\) or \(R\), so it satisfies \(P'(\tau^l,I(\tau^l))\).

Each \(\tau^l\) is incident to the edge \(\tau^l \xrightarrow{+}_{\cM'} (\omega_1^l,\dots,\vec\omega_{I(\tau^{l})-1}^l,\vec\omega_{I(\tau^l)}^{l\prime},\vec\pi^l,\vec\omega_{I(\tau^{l})+1}^l,\dots,\vec\omega_{n-1}^l)= \tau^{l\prime}\) where \(\vec\omega_{I(\tau^l)}^{l\prime}\) is essential and \(\vec\pi^l\) is a partial permutation.
If \(\tau^{l\prime} \xrightarrow{-}_{\cM'} \tau^{l+1}\), since \(\tau^{l+1}\) does not contain a partial permutation, such \(\tau^{l+1}\) is uniquely determined as \(\tau^{l+1} = (\omega_1^l,\dots,\vec\omega_{I(\tau^l)-1}^l,\vec\omega_{I(\tau^l)}^{l\prime},\vec\pi^l\vec\omega_{I(\tau^l)+1}^l,\dots,\vec\omega_{n-1}^l)\).
Therefore, there is at most one path \(p(\tau^0)\) from \(\tau^0\) to an \(\cM'\)-critical cell \(\tau' = \tau'(\tau^0)\), and such a path does not exist if and only if \(\tau^0\) is \(\cM'\)-collapsible. (If \(\tau^0\) is \(\cM'\)-critical, \(p(\tau^0)\) is the length 0 path on \(\tau^0\).)
For this path, we have
\(w(p(\tau^0)) =  [\tau^{k-1\prime}:\tau^k][\tau^{k-1\prime}\tau^{k-1}] \dots [\tau^{1\prime}:\tau^2] [\tau^{1\prime}:\tau^1] [\tau^{0\prime}:\tau^1][\tau^{0\prime}:\tau^0] \)
and
\([\tau^{l\prime}:\tau^{l+1}][\tau^{l\prime}:\tau^{l}] = (-1)^{2I(\tau^{l})+1} = -1\).
(Note that the length \(k\) is also uniquely determined by \(\tau^0\), \(k=k(\tau^0)\).)

We have
\begin{align*}
\delta_n(\tau) = \sum_{\tau \xrightarrow{-}_{\cM'} \tau^0} (-1)^{k(\tau^0)} w\ab(p(\tau^0))[\tau:\tau^0]\tau'(\tau^0)
= \sum_{\tau \xrightarrow{-}_{\cM'} \tau^0} [\tau:\tau^0]\tau'(\tau^0)
\end{align*}
and this equals \(\Phi\delta_n^u(\tau)\) since 
\(\delta_n^u(\tau) = \sum_{\tau^0} [\tau:\tau^0] \tau^0\).
\end{proof}

\section{Anick resolution}
In this section, we briefly recall basic notions of rewriting and Anick resolutions, and then we construct a Morse matching on the normalized bar resolution to give Anick resolutions for Lawvere theories.
\subsection{Abstract rewriting systems}
Let \(X\) be a set and \(\to\) be a binary relation on \(X\).
We write \(\to^+\) for the transitive closure of \(\to\), \(\to^*\) for the reflexive transitive closure of \(\to\), and
\(\leftrightarrow^*\) for the symmetric closure of \(\to^*\).

We say that \(\to\) is \emph{well-founded} if there are no infinite sequences \(a_1 \to a_2 \to \dots\) of elements in \(X\).
We say that \(\to\) is \emph{Church-Rosser} if, for any \(a,b \in X\) such that \(a \leftrightarrow^* b\), there exists \(c \in X\) such that \(a \to^* c\) and \(b \to^* c\).
If \(\to\) is well-founded and Church-Rosser, \(\to\) is called \emph{complete}.
An element \(a \in X\) is \(\to\)-\emph{irreducible} if there is no \(b \in A\) such that \(a \to b\).
An element \(a'\) is a \(\to\)-\emph{normal form} of \(a \in X\) if \(a \to^* a'\) and \(a'\) is \(\to\)-irreducible.

\begin{lemma}\label{lem:rewriting_subrelation}
Let \(\to_1\) be a well-founded and Church-Rosser relation on \(A\) and \(\to_2\) be a well-founded relation on \(A\) satisfying the following condition:
\(a'\) is a \(\to_1\)-normal form of \(a\) if \(a'\) is a \(\to_2\)-normal form of \(a\) for any \(a,a' \in X\).
Then, \(\to_2\) is Church-Rosser and \({\leftrightarrow_2^*} \subset {\leftrightarrow_1^*}\).
\end{lemma}
\begin{proof}
Suppose that \(a \leftrightarrow_2^* b\).
Since \(\to_2\) is well-founded, there exist \(\to_2\)-normal forms \(a',b'\) of \(a,b\), respectively.
By the condition \(a',b'\) are \(\to_1\)-normal forms of \(a,b\), respectively,
and hence \(a' = b'\) since \(\to_1\) is Church-Rosser.
Therefore, \(a \to_2^* a'\) and \(b \to_2^* b'\) imply that \(\to_2\) is Church-Rosser.

Suppose \(a \leftrightarrow_2^* b\) and let \(a'\) be the \(\to_2\)-normal form of \(a,b\).
Then, \(a'\) is also the \(\to_1\)-normal form of \(a,b\) by the condition, so \(a,b \to_1^* a'\), which implies \(a \leftrightarrow_1^* b\).
\end{proof}

\subsection{Review of Anick resolutions}\label{subsec:anick_monoid}
In this subsection, we review Anick resolutions for modules over monoid rings and see how to construct them via Morse matchings following \cite{skoldberg06,joellenbeck05}.
The content of this subsection is not technically used in the subsequent sections, but it helps to understand the main construction of the Anick resolutions for Lawvere theories in the next subsection.
Also, we describe Anick resolutions in terms of rewriting systems as in \cite{brown92}.

Let \(A\) be a set.
For a word \(u \in A^*\), \( |u| \) denotes the length of \(u\).
For two words \(u,v\in A^*\) and an integer \(i\in\{1,\dots,|u|\}\), we say that \(v\) is \emph{the subword of \(u\) at position} \(i\) if \(u = u_1 v u_2\) for some \(u_1,u_2 \in A^*\) such that \(|u_1| = i-1\).

A \emph{semi-congruence} relation on \(A^*\) is a binary relation \(\to\) on \(A^*\) such that, for any \(u,v,u',v' \in A^*\), if \(u \to v\), then \(u'u v' \to u'v v'\).
A semi-congruence relation \(\to\) is a \emph{congruence} relation if it is an equivalence relation.
For any congruence relation \(\sim\) on \(A^*\), the quotient set \(A^*/{\sim}\) has a monoid structure induced by the concatenation on \(A^*\).

For \(R \subset A^*\times A^*\), we write \(\to_R\) for the smallest semi-congruence relation on \(A^*\) containing \(R \subset A^* \times  A^*\), and write \(\sim_R\) for the smallest congruence relation on \(A^*\) containing \(R\).
Also, we write \(R_{\mathrm{L}}\) for the set \(\{l \in A^* \mid (l,r) \in R\}\).

Let \(N\) be a monoid.
A \emph{presentation} of \(N\) is a pair of a set \(A\) of \emph{generators} and a set \(R \subset A^*\times A^*\) of \emph{relators} such that \(N\) is isomorphic to \(A^*/{\sim_R}\).

A set \(R\subset A^* \times A^*\) of relators is called \emph{terminating} if \(\to_R\) is well-founded.
\(R\) is  called \emph{Church-Rosser} or \emph{complete} if \(\to_R\) is Church-Rosser or \emph{complete}, respectively.
Also, a word \(u \in A^*\) is \emph{\(R\)-irreducible} if \(\to_R\) is.
\(u\) is called \emph{\(R\)-reducible} if it is not \(R\)-irreducible, i.e., \(u = u_1 l u_2\) for some \(u_1,u_2 \in A^*\) and \(l \in R_{\mathrm{L}}\).
\(R\) is called \emph{reduced} if,  for any \((l,r)\in R\), \(l\) is \((R\setminus\{(l,r)\})\)-irreducible, and \(r\) is \(R\)-irreducible.

If \(R\) is complete, then each equivalence class \([u]_{\sim_R}\) contains a unique \(R\)-irreducible word, which is called the \emph{\(R\)-normal form} of \(u\).

Suppose that \(N\) is presented by \((A,R)\) for a reduced and complete \(R\).
We define \(n\)-chains as follows.
\begin{definition}\label{def:anick_chain}
The 0-\emph{chain} is the 0-tuple \(()\).
A 1-\emph{chain} is 1-tuple \((a)\) for \(a \in A\).
For \(n > 1\), an \(n\)-tuple \((u_1,\dots,u_n)\) is an \(n\)-\emph{chain} if each \(u_i\) is \(R\)-irreducible,
\((u_1,\dots,u_{n-1})\) is an \((n-1)\)-chain,
\(u_{n-1}u_n\) is \(R\)-reducible, and every proper prefix of \(u_{n-1}u_n\) is \(R\)-irreducible.
\end{definition}
Note that if \((u_1,\dots,u_n)\) is an \(n\)-chain, then \(u_{n-1}u_n\) can be written as \(u_{n-1}u_n = wl\) for some \(w \in A^*\) and \(l \in R_\mathrm{L}\) and cannot be written as \(u_{n-1}u_n = wlw'\) for any \(w, w\in A^*\), \(l \in R_\mathrm{L}\), \(w' \neq 1\).

Anick showed that there is a free resolution of the trivial left \(\bbZ N\)-module \(\bbZ\)
\begin{equation}\label{eqn:anick}
\dots \to F_2 \to F_1 \to F_0 \to \bbZ \to 0
\end{equation}
where \(F_n\) is the free left \(\bbZ N\)-module generated by the set of \(n\)-chains \cite{anick86}.

It is showed in \cite{brown92,skoldberg06,joellenbeck05} that this resolution can be obtained from a Morse matching on the normalized bar resolution,
\[
\dots \xrightarrow{\delta_3} B_2^s \xrightarrow{\delta_2} B_1^s \xrightarrow{\delta_1} B_0^s \xrightarrow{\epsilon} \bbZ \to 0
\]
where \(B_n^s\) is the free left \(\bbZ N\)-module generated by the \(n\)-tuples \((u_1,\dots,u_n)\) of words such that \(u_i \neq 1\) for any \(i=1,\dots,n\), \(\epsilon\) is the canonical augmentation, and
\(\delta_n\) is defined as
\[
\delta_n(u_1,\dots,u_n) = u_1(u_2,\dots,u_n) + \sum_{i=1}^{n-1} (-1)^i (u_1,\dots,u_iu_{i+1},\dots,u_n).
\]
Such a Morse matching \(\cM^s\) is given as follows.
\begin{definition}\label{def:anick_matching}
We define \(\cM^s\) as the set of edges in the graph \(G_{B^s_\bullet}\) of the form
\begin{equation}\label{eqn:edge_morse}
(u_1,\dots,u_{i},u_{i+1}',u_{i+1}'',u_{i+2},\dots,u_n) \to (u_1,\dots,u_n)
\end{equation}
where \(u_{i+1}'u_{i+1}'' = u_{i+1}\), \((u_1,\dots,u_i,u_{i+1}')\) is an \((i+1)\)-chain, and \(i\) is the largest integer less than \(n\) such that \((u_1,\dots,u_i)\) is an \(i\)-chain.
\end{definition}
We mostly omit the proof that \(\cM^s\) is a Morse matching,
but just show that no two edges in \(\cM^s\) share a common target:
Suppose that \(\cM^s\) contains two edges \((u_1,\dots,u_i,u_{i+1}',u_{i+1}'',u_{i+2},\dots,u_n) \to (u_1,\dots,u_n)$ and $(u_1,\dots,u_i,v_{i+1}',v_{i+1}'',u_{i+2},\dots,u_n) \to (u_1,\dots,u_n)\), 
Note that \(u_{i+1}'\) and \(v_{i+1}'\) are prefixes of \(u_{i+1}\).
Since \((u_1,\dots,u_i,u_{i+1}')\) and \((u_1,\dots,u_i,v_{i+1}')\) are \((i+1)\)-chain, \(u_iu_{i+1}' = wl\) and \(u_iv_{i+1}' = w'l'\) for some \(w,w'\in A^*\), \(l,l'\in R_\mathrm{L}\).
If \(u_{i+1}'\) is a proper prefix of \(v_{i+1}'\) (resp. \(v_{i+1}'\) is a proper prefix of \(u_{i+1}'\)), it contradicts to the assumption that every proper prefix \(u_{i}v_{i+1}'\) (resp. \(u_iv_{i+1}'\)) is \(R\)-irreducible. so we have \(u_{i+1}' = v_{i+1}'\), \(u_{i+1}'' = v_{i+1}''\).

Finally, we remark that the \(n\)-chains are exactly the \(\cM^s\)-critical cells.
The target of \eqref{eqn:edge_morse} is not an \(n\)-chain by definition, and the source is not an \((n+1)\)-chain since \(u_{i}u_{i+1}'\) is \(R\)-irreducible.
Conversely, suppose that \((v_1,\dots,v_n)\) is not an \(n\)-chain.
Let \(i\) be the largest integer less than \(n\) such that \((v_1,\dots,v_i)\) is an \(i\)-chain.
If \(v_{i}v_{i+1}\) is \(R\)-reducible, then \(v_iv_{i+1} = wlw'\) for some \(w,w'\in A^*\), \(l \in R_{\mathrm{L}}\), \(w'\neq 1\), so \((v_1,\dots,v_i,wl,w',v_{i+2},\dots,v_n) \to (v_1,\dots,v_n)\) is in \(\cM^s\).
Otherwise, \((v_1,\dots,v_n) \to (v_1,\dots,v_iv_{i+1},\dots,v_n)\) is in \(\cM^s\).
Thus \((v_1,\dots,v_n)\) is not \(\cM^s\)-critical.

% 
%
%Finally, we shall see that 3-chains corresponds to \emph{critical branchings} (or \emph{critical pairs}) developed in the theory of rewriting systems \cite{book93,baader98}.
%
%For any \(3\)-chain \((u_1,u_2,u_3)\), there exists a unique \(2\)-tuple of relations \(((l_1, r_1),(l_{2}, r_{2})) \in R^{2}\) such that \(l_i\) is a subword of \(u_i u_{i+1}\) for each \(i=1,2\).
%This is because, if \(l_i\), \(l_i'\) are subwords of \(u_iu_{i+1}\) for some \((l_i,r_i),(l_i',r_i')\in R\), then \(u_iu_{i+1} = vl_i = v'l_i'\) for some words \(v,v'\) since any proper subword of \(u_iu_{i+1}\) is \(R\)-irreducible, and hence \((l_i,r_i) = (l_i',r_i')\) by the reducedness of \(R\).
%
%For \((l_1,r_1),(l_2,r_2) \in R\), 
%if \(l_1 = uv\) and \(l_2 = vw\) for some \(u,v,w\) with \(v\neq \epsilon\), we have \(uvw \to_R r_1 w\) and \(uvw \to_R ur_2\).
%In such a case, we call the diagram \(r_1 w \xleftarrow{(l_1,r_1)} uvw \xrightarrow{(l_2,r_2)} ur_2\) a \emph{critical branching}.
%
%It is not too hard to show that there is a one-to-one correspondence between \(3\)-chains and critical branchings. For example,
%\begin{example}
%Consider \(\varsign = \{a,b\}\) and \(R = \{\rho_1=(abb,\epsilon), \rho_2=(bba,\epsilon)\}\).
%It is obvious that \(R\) is reducible.
%
%We have three critical branchings \(a \xleftarrow{\rho_1} abba \xrightarrow{\rho_2} a\),
%\(ba \xleftarrow{\rho_1} abbba \xrightarrow{\rho_2} ab\),
%and \(bb \xleftarrow{\rho_1} bbabb \xrightarrow{\rho_2} bb\),
%and also
%three \(3\)-chains
%\((a,bb,a)\), \((a,bb,ba)\), \((b,ba,bb)\).
%\end{example}

\subsection{Terms}
Let \(\varsign\) be an \(\varsorts\)-sorted signature.

For a term \(t\) and a finite sequence \(p\) positive integers,
we define a term \(t|_p\) as follows:
\(t|_\epsilon = t\) for the empty sequence \(\epsilon\), and \(t|_{ip'} = t_i|_{p'}\) if \(t = f(t_1,\dots,t_n)\) and \(i \in \{1,\dots,n\}\). Otherwise, \(t|_{ip'}\) is not defined.
Also, for \(m\)-tuple of terms \(\vec t = \ab(t_1,\dots,t_m)\), define \(\vec t|_p = t_i|_{p'}\) if \(p = ip'\), \(i \in \{1,\dots,m\}\). Otherwise, \(\vec t|_p\) is not defined.
We define \(\Pos(\vec t)\) for the set of sequences \(p\) such that \(\vec t|_p\) is defined.
We call elements in \(\Pos(\vec t)\) \emph{positions} in \(\vec t\).
We say that \(t'\) is a \emph{subterm} of \(t\) if \(t' = t|_p\) for some position \(p\) in \(t\).

From now, we identify a tuple terms \(\vec t\) with the morphism \((\Gamma \mid_\emptyset \vec t)\) in \(\syn{\varsign}\).

We say that a term \(t' : \vec X' \to Y\) is a \emph{generalized subterm} of a term \(t : \vec X \to Y \), written \(t' \legensub t\), if there exists a subterm \(t'' : \vec X \to Y\) of \(t\) and \(\vec s : \vec X \to \vec X'\) such  that \(t'' = t'\vec s\).
In other words, if \(\Var(t) = \{x_1,\dots,x_n\}\) and \(\Var(t') = \{y_1,\dots,y_m\}\), then \(t' \legensub t\) if and only if there exists \(s_1,\dots,s_m\) such that \(t'[s_1/y_1,\dots,s_m/y_m]\) is a subterm of \(t\).
Also, \(t'\) is a \emph{proper} generalized subterm of \(t\), written \(t' \ltgensub t\), if \(t'\) is a generalized subterm of \(t\) and \(t' \neq t\).
Note that \(\legensub\) is a partial order and any term has only finitely many lower bound with respect to \(\legensub\) (modulo \(\sim_\alpha\)).

A \emph{rewrite rule} is a morphism \((\Gamma \mid_\emptyset (l,r))\) in \(\syn{\varsign}\) for terms \(l,r\) of the same sort such that \(\Var(r)\subset \Var(l)\),
and we write \(\Gamma \mid_\emptyset l \to r\) or just \(l \to r\) for the rewrite rule.
A \emph{term rewrite system (TRS)} is a set of rewrite rules.

For a rewrite rule \(\Gamma \mid_\emptyset l\to r\) (\(\Gamma = x_1,\dots,x_n\)) and terms \(t, t'\), we say that \(t\) is \emph{rewritten to} \(t'\) by \(\Gamma\mid_\emptyset l\to r\) if there exist a rewrite context \(C\) and \(n\) terms \(\vec s = \ab(s_1,\dots,s_n)\) such that \(t = C[l\vec s]\) and \(t' = C[r\vec s]\).
We write \(t \to_R t'\) if \(t\) is rewritten to \(t'\) by some rule in \(R\).

Note that \(t\) can be written by \(l \to r\) if and only if \(l\) is a generalized subterm of \(t\).
Also, it is not difficult to show that, 
by regarding \(R\) as a set of equations,
\(\approx_R\) coincides with \(\leftrightarrow_R^*\), the reflexive symmetric transitive closure of \(\to_R\).

We can check that \(\to_R\) is compatible with the equivalence relation \(\sim_\alpha\), i.e., if \(t\sim_\alpha t'\), \(s\sim_\alpha s'\) and \(t \to_R s\), then \(t' \to_R s'\).
So, we can define \(\to_R\) for morphisms \((\Gamma \mid_\emptyset t)\), \((\Gamma \mid_\emptyset s)\) in \(\syn{\varsign}\) as \((\Gamma \mid_\emptyset t) \to_R (\Gamma \mid_\emptyset s) \iff t \to_R s\).

We say that \(R\) is \emph{terminating}, \emph{Church-Rosser}, or \emph{complete} if \(\to_R\) is well-founded, Church-Rosser, or complete, respectively, and say \(R\)-\emph{irreducible} or \(R\)-\emph{normal forms} for \(\to_R\)-irreducible or \(\to_R\)-normal forms, respect
For any two terms \(t,s\), \(t \approx_R s\) if and only if the \(R\)-normal forms of \(t\) and \(s\) coincide (see \cite[Theorem 2.19]{baader98}, for example).
We say that a TRS \(R\) is \emph{reduced} if, for  any rule \(l \to r \in R\), \(l\) is \(R\setminus\{l \to r\}\)-irreducible and \(r\) is \(R\)-irreducible.

\begin{lemma}
For any terminating TRS \(R\),
\(R\) is Church-Rosser if and only if for any term \(t\), there exists a unique \(R\)-irreducible term \(\hat t\) such that \(t \to_R^* \hat t\).
\end{lemma}
\begin{proof}
Suppose that \(R\) is Church-Rosser.
If \(t \to_R^* s_1\) and \(t \to_R^* s_2\) for a term \(t\) and \(R\)-irreducible terms \(s_1,s_2\), then, by the Church-Rosser property, we have \(s_1 = s_2\).
Conversely, suppose that for any term \(t\), there exists a unique \(R\)-irreducible term \(\hat t\) such that \(t \to_R^* \hat t\).
For terms \(t_1, t_2\) with \(t_1 \approx_R t_2\), let \(\hat t_1, \hat t_2\) be the \(R\)-irreducible terms such that \(t_1 \to_R^* \hat t_1\) and \(t_2 \to_R^* \hat t_2\).
Then, \(\hat t_1 = \hat t_2\) is easily proved by the induction on the length of \(t_1 \leftrightarrow_R^* t_2\).
\end{proof}

\begin{proposition}
For any complete TRS \(R\), there exists a reduced complete TRS \(R'\) such that \({\leftrightarrow_R^*} = {\leftrightarrow_{R'}^*}\).
\end{proposition}
\begin{proof}
Let \(R\) be a complete TRS and
\(R'' = \{l \to \hat r \mid l \to r \in R, \hat r \text{ is the \(R\)-normal form of } r\}\).
Since \({\to_{R''}} \subset {\to_R^+}\) and \(R\) is terminating, \({\leftrightarrow_{R''}^*}\subset {\leftrightarrow_{R}^*}\) and
\(R''\) is terminating.
Any \(R''\)-irreducible term \(t\) is \(R\)-irreducible since \(R\) and \(R''\) have the same set of left-hand sides of rules.
Then, we can apply \cref{lem:rewriting_subrelation} to conclude that \(R''\) is Church-Rosser.

We show \({\leftrightarrow_R^*}\subset {\leftrightarrow^*_{R''}}\) by applying \cref{lem:rewriting_subrelation} again.
\(R\)-irreducible terms are also \(R''\)-irreducible since \(\to_{R''} \subset \to_R^+\).
We show that \(\hat t\) is the \(R''\)-normal form of \(t\) if \(\hat t\) is the \(R\)-normal form of \(t\) by the well-founded induction on \(t\) with respect to the well-founded relation \(\to_{R''}\).
If \(t\) is \(R\)-irreducible, then \(t = \hat t\).
Otherwise, there exists \(t'\) such that \(t \to_{R''} t'\).
Since \(R\) is complete and \(t \to_R^* t'\), we have \(t' \to_R^* \hat t\).
By the induction hypothesis, we have \(t' \to_{R''}^* \hat t\), so \(t \to_R^* \hat t\).
Therefore, we can apply \cref{lem:rewriting_subrelation} and obtain \({\leftrightarrow^*_{R}} \subset {\leftrightarrow^*_{R''}}\), hence \({\leftrightarrow^*_R} = {\leftrightarrow^*_{R''}}\).

Let \(R' = \{l \to r \in R'' \mid l \text{ is \((R''\setminus\{l\to r\})\)-irreducible}\}\).
By definition, \(R'\) is reduced.
We prove that \(R'\) is complete and \({\leftrightarrow^*_{R''}} = {\leftrightarrow^*_{R'}}\).
Since \({\to_{R'}}\subset {\to_{R''}}\) and \(R''\) is terminating, \(R'\) is also terminating, \(\approx_{R'}\subset \approx_{R''}\) and any \(R''\)-irreducible terms are \(R'\)-irreducible.
We show that any \(R''\)-reducible terms are \(R'\)-reducible.
Suppose that \(t\) is \(R''\)-reducible and can be rewritten by some \(l\to r \in R''\setminus R'\).
Then, there exists \(l' \to r' \in R''\setminus{\{l \to r\}}\) that can rewrite \(l\).
Since \(l'\) is a proper generalized subterm of \(l\),
we have \(l' \to r' \in R'\) and \(t\) can be rewritten by \(l' \to r'\).

We show that \(\hat t\) is the \(R'\)-normal form of \(t\) if \(\hat t\) is the \(R''\)-normal form of \(t\) by the well-founded induction on \(t\) with respect to the well-founded relation \(\to_{R''}\).
If \(t\) is \(R''\)-irreducible, then \(t = \hat t\).
Otherwise, \(t\) is \(R'\)-reducible, so there exists \(t'\) such that \(t \to_{R'} t'\) and \(t \to_{R''} t'\).
Since \(R''\) is complete, \(t' \to_{R''}^* \hat t\).
Since \(R'\) is complete and \(t \to_{R''}^* t'\), we have \(t' \to_{R'}^* \hat t\).
By the induction hypothesis, we have \(t' \to_{R'}^* \hat t\), so \(t \to_{R''}^* \hat t\).
Therefore, we can apply \cref{lem:rewriting_subrelation} and conclude that \(R'\) is Church-Rosser and \({\leftrightarrow^*_{R''}} \subset {\leftrightarrow^*_{R'}}\), hence \({\leftrightarrow^*_{R''}} = {\leftrightarrow^*_{R'}}\).
\end{proof}

A \emph{unifier} of a pair \((\vec t, \vec t')\) of two morphisms \(\vec t : X_1\timesdots X_n \to Y_1\timesdots Y_m\), \(\vec t' : X_1'\timesdots X_k' \to Y_1\timesdots Y_m \) in \(\syn{\varsign}\) is a pair \((\vec s, \vec s')\) of two morphisms \(\vec s : \vec Z \to X_1\timesdots X_n\), \(\vec s' : \vec Z \to X_1'\timesdots X_k'\) in \(\syn{\varsign}\) such that \(\vec t \vec s = \vec t' \vec s'\).
A unifier is the \emph{most general} if \(\vec t\vec s\) is essential and, for any other unifier \((\vec r , \vec r')\), there exists \(\vec u\) such that \(\vec r = \vec s\vec u\) and \(\vec r' = \vec s'\vec u\).

Church-Rosser property of a terminating TRS can be easily checked by using the notion of critical pairs.
\begin{definition}
Let \((\Gamma_i \mid_\emptyset l_i \to r_i) \in R\) (\(i=1,2\)).
If there is \(p \in \Pos(l_1)\) such that \(l_1|_p\) is not a variable and if \((\vec s_1,\vec s_2)\) is the most general unifier of \((l_1|_p, l_2)\), then we call the pair \((r_1 \vec s_1, t)\) a \emph{critical pair} where \(t\) is the term obtained from \(l_1\vec s_1\) by replacing the subterm at position \(p\) with \(r_2 \vec s_2\).
\end{definition}
\begin{lemma}\label{lem:cplem}
A terminating TRS \(R\) is Church-Rosser if and only if, for any critical pair \((t,t')\) of \(R\), there exists a term \(s\) such that \(t \to_R^* s\) and \(t' \to_R^* s\).
\end{lemma}
\begin{proof}
See \cite[Corollary 6.2.5]{baader98}, for example.
\end{proof}

\subsection{Observations}
In this subsection, we observe that some notions in \ref{subsec:anick_monoid} cannot be simply extended to Lawvere theories.

%For a term \(t\), an \(n\) rewrite rules \(l_1\to r_1,\dots,l_n \to r_n\), and \(n\)  positions \(p_1,\dots,p_n\) in \(t\), we say that \((t, (l_1\to r_1,p_1),\dots,(l_n \to r_n,p_n))\) is an \(n\)-\emph{branching} if \(t\) can be rewritten by \(l_i \to r_i\) at \(p_i\) for each \(i=1,\dots,n\).

As an example, consider \(\varsorts = \{X\}\), \(\varsign = \{+, 0\}\) with \(\alpha(+) = (XX,X)\), \(\alpha(0) = (\epsilon,X)\), and the TRS \(R\) consisting of
\[
\rho_1 = (x\mid_\emptyset x+0 \to x), \quad \rho_2 = (x\mid_\emptyset 0 + x \to x)
\]
and we work on the category \(\syn{\varsign, R}\).
Obviously \(R\) is reduced.
\(R\) is terminating since the number of 0s in a term strictly decreases by applying a rule in \(R\), and Church-Rosser property of \(R\) can be proved by \cref{lem:cplem}.

We can identify \((\Gamma \mid_R t)\) in \(\syn{\varsign, R}\) with \((\Gamma \mid_\emptyset \hat t)\) in \(\syn{\varsign}\) where \(\hat t\) is the \(R\)-normal form of \(t\).

We shall say that a 1-cell \( (\omega) \) is a 1-chain if \(\omega = (x_1,x_2 \mid_\emptyset x_1 + x_2)\) or \(\omega = (\mid_\emptyset 0)\), which contains exactly one symbol in \(\varsign\).
Imitating \cref{def:anick_chain}, we may define a 2-cell \(((x_1,\dots,x_m\mid_\emptyset t), (y_1,\dots,y_n\mid_\emptyset (s_1,\dots,s_m)))\) to be a 2-chain if \((x_1,\dots,x_m\mid_\emptyset t)\) is a 1-chain, \(t' = t[s_1/x_1,\dots,s_m/x_m]\)  is \(R\)-reducible, and every proper generalized subterm of \(t'\) is \(R\)-irreducible.
This means that \(t = l\) for some \(\rho = (\Gamma \mid_\emptyset l \to r)\) in \(R\).
So, for our example, we have two 2-chains, \(((x_1,x_2 \mid_\emptyset x_1 + x_2), (x \mid_\emptyset (x,0)))\) and \(((x_1,x_2 \mid_\emptyset x_1 + x_2), (x \mid_\emptyset (0,x)))\).

Similarly, we may want to define an \(n\)-chain as an \(n\)-cell \((\omega_1,\vec\omega_2,\dots,\vec\omega_n)\) such that \((\omega_1,\vec\omega_2,\dots,\vec\omega_{n-1})\) is an \((n-1)\)-chain, \(t\) is \(R\)-reducible, and every proper generalized subterm of \(t\) is \(R\)-irreducible where \(\omega_{n-1} \omega_n = (\Gamma \mid_\emptyset t)\).
However, this definition does not give us a partial matching \(\cM\) as in \cref{def:anick_matching}.
For example, consider the 2-cell \(\tau=((x_1,x_2 \mid_\emptyset x_1 + x_2), (\mid_\emptyset (0,0)))\) and the two 3-cells \(\tau_1=((x_1,x_2 \mid_\emptyset x_1 + x_2), (x \mid_\emptyset (0,x)), (\mid_\emptyset 0))\) and \(\tau_2 = ((x_1,x_2 \mid_\emptyset x_1 + x_2), (x \mid_\emptyset (x,0)), (\mid_\emptyset 0))\), illustrated in \cref{fig:left_right_unit2_3}.
Since \(\tau\) is not a 2-chain and \((x_1,x_2 \mid_\emptyset x_1 + x_2)\) is a 1-chain, and since \((x\mid_\emptyset (0,x))(\mid_\emptyset 0) = (x\mid_\emptyset (x,0))(\mid_\emptyset 0) = (\mid_\emptyset (0,0))\), \(\cM\) should contain both \(\tau_1 \to \tau\) and \(\tau_2 \to \tau\) according to \cref{def:anick_matching}, but then \(\cM\) is not a partial matching.

To solve this problem, we will define \(n\)-chains in the way that one of \(\tau_1,\tau_2\) is a 3-chain, and the other one is adjacent to \(\tau\) in \(\cM\).

\begin{figure}[htb]
\centering
%\begin{minipage}{0.4\linewidth}
%
% Left and right unit 2
\begin{tikzpicture}
\node[circle,draw] (root) at (2,2) {+};
\node[circle,draw] (r1) at (1, 0) {\(0\)};
\node[circle,draw] (r2) at (3,1) {\(0\)};
\coordinate (interm) at (1,1);

\node at (2,-0.5) {$\tau_1$};

\draw (root) -- (interm);
\draw (root) -- (r2);
\draw (interm) -- (r1);

\draw[dashed] (0.5, 1.5) -- (3.5, 1.5);
\draw[dashed] (0.5, 0.5) -- (3.5, 0.5);
\end{tikzpicture}
~
% Left and right unit 3
\begin{tikzpicture}
\node[circle,draw] (root) at (2,2) {+};
\node[circle,draw] (r1) at (1, 1) {\(0\)};
\node[circle,draw] (r2) at (3,0) {\(0\)};
\coordinate (interm) at (3,1);

\node at (2,-0.5) {$\tau_2$};

\draw (root) -- (interm);
\draw (root) -- (r1);
\draw (interm) -- (r2);

\draw[dashed] (0.5, 1.5) -- (3.5, 1.5);
\draw[dashed] (0.5, 0.5) -- (3.5, 0.5);
\end{tikzpicture}
%\subcaption{\label{fig:left_right_unit2_3}}
%\end{minipage}
%\begin{minipage}{0.15\linewidth}
%% Left and right unit
%\begin{tikzpicture}[level distance=1cm]
%\pgfdeclarelayer{bg}
%\pgfsetlayers{bg,main}
%  \node[circle,draw] (root) {+}
%	child {node[circle,draw] (r1) {\(0\)}}
%	child {node[circle,draw] (r2) {\(0\)}};
%
%	\tikzset{
%	band/.style={line join=round,draw opacity=.8,line width=15pt},
%	}
%	\begin{pgfonlayer}{bg}
%	\draw[band,magenta,line cap=round] (root.center) -- (r2.center);
%	\draw[band,magenta] (root.center) -- (r1);
%	\draw[band,cyan,line cap=round,transform canvas={shift={(0.1,0)}}]
%		(r1.center) -- (root.center);
%	\draw[band,cyan,transform canvas={shift={(0.1,0)}}]
%		(r2) -- (root.center);
%
%	\end{pgfonlayer}
%
%	\coordinate (lab) at ($(r1.south)+(0,-0.4)$);
%	\fill[fill=magenta] ($(lab)+(0,-0.4)$) rectangle ++(0.5,-0.3);
%	\node at ($(lab)+(0.8,-0.55)$) {\(\rho_1\)};
%	\fill[fill=cyan] ($(lab)+(0,-0.8)$) rectangle ++(0.5,-0.3);
%	\node at ($(lab)+(0.8,-0.95)$) {\(\rho_2\)};
%\end{tikzpicture}
%\subcaption{\label{fig:left_right_unit1}}
%\end{minipage}
%\label{fig:left_right_unit}
\caption{}\label{fig:left_right_unit2_3}
\end{figure}

\subsection{Morse matching}

We construct a Morse matching \(\cM\) on the normalized bar resolution \(B_\bullet\) of an \(\varsorts\)-sorted Lawvere theory \(\varlaw\).
From now, we call \(\cM'\)-critical cells (i.e., generators of \(B_\bullet\)) just \emph{cells}.

Suppose that \(\varlaw\) is presented by \((\Sigma,R)\) where \(R\) is a reduced and complete TRS over \(\Sigma\).
Then, a morphism \(\vec\omega: \vec \varsort \to \vec{\varsort[1]}\) in \(\varlaw\) uniquely corresponds to \((\Gamma \mid_R \vec s)\) in \(\syn{\Sigma,R}\), and to \((\Gamma \mid_\emptyset \vec t)\) in \(\syn{\Sigma}\) where \(\vec t\) is the \(R\)-normal form of \(\vec s\).
Also, if \(\vec \omega\) is essential, \(\Gamma\) is uniquely determined as the list of variables in \(\vec t\).
So, we just write \(\vec t\) for essential \(\vec\omega = (\Gamma \mid_R \vec t)\).
Also, when we write \(\vec t \vec t'\) for \(R\)-irreducible \(\vec t\), \(\vec t'\), we mean it as the morphism \((\Gamma \mid_\emptyset \vec t)(\Gamma'\mid_\emptyset \vec t')\) in \(\syn{\Sigma}\), which may be \(R\)-reducible.
We write \(\vec t\star \vec t'\) for the \(R\)-normal form of \(\vec t \vec t'\).
In particular, with this notation, we have
\[
d_j(t_1,\vec t_2,\dots,\vec t_n) = (t_1, \vec t_2, \dots \vec t_{j-1},\vec t_j\star \vec t_{j+1},\vec t_{j+2},\dots,\vec t_n)
\]
for \(0 < j < n\). (Recall \ref{subsec:unnormalized} for the definition of \(d_j\).)

Note that \(\vec t\) contains at least one function symbol if \(\vec\omega_i\) is not a partial permutation by \cref{lem:term_effmor}.
For \(f \in \Sigma\), we also identify \(t = f(x_1,\dots,x_n)\) with \(f\).
In particular, we write \(t \in \Sigma\) if \(t\) is of the form \(f(x_1,\dots,x_n)\) for some \(f \in \Sigma\), and \(t \notin \Sigma\) if not.
%For \(R\)-irreducible tuples of term-in-contexts \(t_1,\vec t_2,\dots,\vec t_n\) that are composable as \(t_1\vec t_2\dots \vec t_n\), we write \((t_1\mid \vec t_2\mid \dots \mid \vec t_n)\) for the cell \((\omega_1,\vec\omega_2,\dots,\vec\omega_n)\) where \(\omega_1 = (\Gamma_1 \mid_R t_1)\), \(\vec\omega_i = (\Gamma_i \mid_R \vec t_i)\) for \(i=2,\dots,n\).

Let \(R_\mathrm{L} = \{l \mid \exists r.~l \to r \in R\}\),
\(P(\vec t) = \ab\{(p,l) \mid p \in \Pos\ab(\vec t), l \in R_{\mathrm{L}}, \vec t|_p = l\vec s\text{ for some }\vec s\}\).

We fix a well-founded total order \(<\) on \(R_\mathrm{L}\), i.e., a total order \(<\) such that there is no infinite increasing sequence in \(R_{\mathrm{L}}\).
We extend \(<\) to \(\bbZ_{>0}^*\times R_\mathrm{L}\) as \((p,l) < (p',l')\) if \(p\) is smaller than \(p'\) in the lexicographic order on \(\bbZ_{>0}^*\), or \(p = p'\) and \(l < l'\).
Write \(\mathbf{p}(\vec t)\) for the maximum in \(P(\vec t) \subset (\bbZ_{>0}^*\times R_\mathrm{L})\cup \{-\infty\}\) where \(-\infty\) is considered as the minimum, that is, \(\mathbf{p}(\vec t) = -\infty\) if and only if \(\vec t\) is \(R\)-irreducible.

The following lemma is obvious by definition.
\begin{lemma}
For any \(\vec t,\vec s\), \(\mathbf{p}(\vec t) \le \mathbf{p}(\vec t \vec s)\).
\end{lemma}

\begin{definition}
The 0-cell \(()_{\varsort}\) is a 0-chain for any \(\varsort\).
A 1-cell \((t_1)\) is an 1-chain if \(t_1 \in \varsign\).
For \(n > 1\), an \(n\)-cell \(\tau = (t_1,\vec t_2,\dots,\vec t_n)\) is an \(n\)-chain
if \(\tau'=(t_1,\vec t_2,\dots,\vec t_{n-1})\) is an \((n-1)\)-chain, \(\bfp(t_1\vec t_2\dots \vec t_{n-1}) < \bfp(t_1\vec t_2\dots \vec t_n)\), and there exists \(\vec s\) such that \((\vec t_n,\vec s)\) is the most general unifier of \(((t_1\vec t_2\dots \vec t_{n-1})|_p, l)\) for \((p,l) = \bfp(t_1\vec t_2\dots \vec t_n)\).
\end{definition}
\begin{example}\label{ex:chain}
Consider the TRS in the previous subsection.
Then, \(R_\mathrm{L} = \{x + 0, 0 + x\}\), and we set \(x + 0 < 0 + x\).
We have \(\bfp(x + 0)  < \bfp(0 + 0)\), so \(\tau_1 = ((x_1,x_2 \mid_\emptyset x_1 + x_2), (x\mid_\emptyset (x,0)), (\mid_\emptyset 0))\) is a 3-chain.
On the other hand, \(\bfp(0 + x) = \bfp(0 + 0)\), so \(\tau_2 = ((x_1,x_2 \mid_\emptyset x_1 + x_2), (x\mid_\emptyset (0,x)), (\mid_\emptyset 0))\) is not a 3-chain.
\end{example}

\begin{lemma}
There are one-to-one correspondences between the set of 0-chains and \(\varsorts\), between the set of 1-chains and \(\Sigma\), and between the set of 2-chains and \(R\).
\end{lemma}
\begin{proof}
Obvious from the definition of chains.
\end{proof}
\begin{lemma}
If \(\varsorts, \varsign, R\) are finite sets, then the set of \(n\)-chains is finite for each \(n\).
\end{lemma}
\begin{proof}
The sets of 0- and 1-chains are obviously finite.
Suppose that the set of \((n-1)\)-chains is finite.
Let \((t_1,\vec t_2 ,\dots ,\vec t_{n-1})\) be an \((n-1)\)-chain.
Since the number of choices of \(l \in R_\mathrm{L}\), \(\vec t_n\), and \(p \in \Pos\ab(t_1\vec t_{2}\dots \vec t_{n-1})\) such that \((\vec t_n,\vec s)\) is the most general unifier of \(\ab(\ab(t_1\vec t_{2}\dots \vec t_{n-1})|_p,l)\) for some \(\vec s\) is finite, there are only finite number of \(\vec t_n\) such that \((t_1,\vec t_2 ,\dots ,\vec t_n)\) is an \(n\)-chain.
\end{proof}

For an \(n\)-cell \(\tau = (t_1,\vec t_2,\dots,\vec t_n)\), define \(L(\tau)\) to be the largest integer in \(\{1,\dots,n\}\) such that \((t_1,\vec t_2,\dots,\vec t_{L(\tau)-1})\) is an \((L(\tau)-1)\)-chain.

\begin{definition}
We define \(\cM = \cM_{\Sigma,R}\) as the set of edges of the form
\begin{equation}\label{eqn:morse_edge_lawvere}
(t_1,\vec t_2,\dots,\vec t_{L(\tau)}', \vec t_{L(\tau)}'', \dots, \vec t_n) \to (t_1, \vec t_2, \dots, \vec t_n) = \tau, \quad (\vec t_{L(\tau)} = \vec t_{L(\tau)}'\vec t_{L(\tau)}'')
\end{equation}
that satisfy one of the following conditions:
\(L(\tau)=1\) and \(t_{1}' \in \varsign\), or
\(L(\tau)>1\) and
there exists \(\vec s\) such that \((\vec t_{L(\tau)}',\vec s)\) is the most general unifier of \((t|_p, l)\) for \(t=t_1\vec t_2\dots \vec t_{L(\tau)-1}\) where \((p,l) = \bfp(t \vec t_{L(\tau)})\).
We say that \(\vec t_{L(\tau)} = \vec t_{L(\tau)}' \vec t_{L(\tau)}''\) is the \emph{\(t\)-decomposition}.
\end{definition}

\begin{example}
Again, consider the TRS in the previous subsection and the order on \(R_\mathrm{L}\) in Example \ref{ex:chain}.
Let \(\tau = ((x_1,x_2\mid_\emptyset x_1 + x_2), (\mid_\emptyset (0,0)))\).
Then, we can check \(L(\tau) = 2\) and \(\ab(\mid_\emptyset (0,0)) = (x\mid_\emptyset (0,x)) (\mid_\emptyset 0) \) is the \((0+0)\)-decomposition.
Therefore, \(\cM\) contains the edge \(\tau_2 = ((x_1,x_2\mid_\emptyset x_1 + x_2), (x\mid_\emptyset (0,x)), (\mid_\emptyset 0)) \to ((x_1,x_2\mid_\emptyset x_1 + x_2), (\mid_\emptyset (0,0)))\).
\end{example}

By construction, no two distinct edges in \(\cM\) share the common source.
Also, since the choice of \((\vec t_{L(\tau)}',\vec t_{L(\tau)}'')\) is unique, no two distinct edges in \(\cM\) share the common target.

\begin{lemma}
For any edge \eqref{eqn:morse_edge_lawvere} in \(\cM\),
the cell \(\tau' = (t_1,\vec t_2,\dots,\vec t_{L(\tau)-1}, \vec t_{L(\tau)}')\) is an \(L(\tau)\)-chain, and there do not exist \(\vec s,\vec s'\) such that \(\vec t_{L(\tau)}'' = \vec s \vec s'\) and \((t_1,\vec t_2,\dots,\vec t_{L(\tau)-1}, \vec t_{L(\tau)}', \vec s')\) is an \((L(\tau)+1)\)-chain.
\end{lemma}
\begin{proof}
It is straightforward that \(\tau'\) is an \(L(\tau)\)-chain by definition.
Assume for contradiction that there exist \(\vec s,\vec s'\) such that \(\vec t_{L(\tau)}'' = \vec s \vec s'\) and \(\tau'' = (t_1,\vec t_2,\dots,\vec t_{L(\tau)-1}, \vec t_{L(\tau)}', \vec s')\) is an \((L(\tau)+1)\)-chain.
Since \(\bfp(t\vec t_{L(\tau)}'\vec s) \le\bfp(t\vec t_{L(\tau)}) = \bfp(t\vec t_{L(\tau)}') \) and \(\bfp(t\vec t_{L(\tau)}') < \bfp(t\vec t_{L(\tau)}'\vec s)\), it contradicts.
\end{proof}
From this lemma, we can see that no two edges in \(\cM\) are adjacent.
So, we have
\begin{lemma}
\(\cM\) is a partial matching on \(B_\bullet\).
\end{lemma}

To show that \(\cM\) is a Morse matching, we define a relation \(\succ\) on cells that is well-founded and satisfies \(\tau \xrightarrow{+-}_\cM \tau' \implies \tau \succ \tau'\) for any \(\tau,\tau'\).
For a cell \(\tau=(t_1,\vec t_2,\dots,\vec t_n)\), let \(T(\tau) = t_1\vec t_2\dots\vec t_n\).
\begin{definition}\label{def:wford}
Define a relation \(\succ\) on \(n\)-cells as follows.
For two \(n\)-cells \(\tau = (t_1,\vec t_2,\dots,\vec t_n)\) and \(\tau' = (t_1',\vec t_2',\dots,\vec t_n')\), \(\tau \succ \tau'\) iff either
\begin{enumerate}
	\item \(T(\tau) \succ_{\mathrm{gs}} T(\tau') \lor T(\tau) \to_R^+ T(\tau')\),
	\item \(T(\tau) = T(\tau')\), \(t_1 \notin \varsign\), and \(t_1' \in \varsign\), or
	\item \(T(\tau) = T(\tau')\) and there exists \(i \in \{1,\dots,n-1\}\) such that \((t_1,\vec t_2,\dots,\vec t_i)\) is an \(i\)-chain and equal to \((t_1',\vec t_2',\dots,\vec t_i')\), and either
	\begin{enumerate}
		\item \(\bfp(t_1\vec t_2\dots \vec t_{i+1}) < \bfp(t_1'\vec t_2'\dots \vec t_{i+1}')\), or
		\item \(\bfp(t_1\vec t_2\dots \vec t_{i+1}) = \bfp(t_1'\vec t_2'\dots \vec t_{i+1}')\) and \(\vec t_{i+1}' = \vec t_{i+1}\vec s\) for some \(\vec s\). 
	\end{enumerate} 
\end{enumerate}
\end{definition}

\begin{lemma}
\(>\) is well-founded.
\end{lemma}
\begin{proof}
First, \(\succ'\) be the union of the two relations \(\succ_{\mathrm{gs}}\), \(\to_R^+\) between terms.
Note that \(\succ_{\mathrm{gs}}\) and \(\to_R^+\) are well-founded.
In general, the union \({>_1}\cup {>_2}\) of two well-founded relation \(>_1\) and \(>_2\) on a set \(A\) is well-founded if \(\exists y.~x >_1 y >_2 z\) implies \(\exists y.~ x >_2 y >_1^* z\) for any \(x,z \in A\) where \(>_1^*\) is the reflexive transitive closure of \(>_1\) \cite{bachmair-dershowitz86}.
This applies to our situation because if \(t \succ_{\mathrm{gs}} t'\) for some \(t'\), there is \(\vec s\) such that \(t'\vec s\) is a subterm of \(t\), and moreover, if \(t' \to_R^+ t''\), then \(t\to_R^+ t'''\) where \(t'''\) is the term obtained by replacing the subterm \(t'\vec s\) in \(t\) with \(t''\vec s\).

Assume that there exists an infinite sequence \(\tau_1 \succ \tau_2 \succ \dots\) of \(n\)-cells.
Since \(\succ'\) is well-founded, there exists \(k\) such that for any \(l \ge k\), \(\tau_l \succ \tau_{l+1}\) holds by 2 or 3 in \cref{def:wford}.

Let \(\tau_l = (t_{l,1} , \vec t_{l,2} , \dots , \vec t_{l,n})\).
If \(\tau_l \succ \tau_{l+1}\) holds by 2, \(t_{l+1,1} \notin \varsign\), but then \(\tau_{l+1} \succ \tau_{l+2}\) cannot hold by either 2 or 3 since \((t_{l+1,1})\) is not a 1-chain.
Therefore, \(\tau_l\succ \tau_{l+1}\) must hold by 3 for any \(l \ge k\).
Since the relation \(<\) on \(\Pos(T(\tau))\times R_\mathrm{L}\) is well-founded,
\(\vec t_{k''} = \vec t_{k''+1}\vec s_1\ = \vec t_{k''+2}\vec s_2\vec s_1 = \dots\) for some \(k'' \ge k'\) and some essential non-identity \(\vec s_1, \vec s_2, \dots\), but this is impossible since every \(\vec s_i\) has at least one function symbol  while the numbers of function symbols in \(\vec s_1, \vec s_2\vec s_1, \dots\) are bound above by the number of function symbols in \(T(\tau)\).
\end{proof}

\begin{lemma}
If \(\tau \xrightarrow{+}_\cM \tau'' \xrightarrow{-}_\cM \tau'\) and \(\tau'\) is \(\cM\)-redundant, then \(\tau \succ \tau'\).
\end{lemma}
\begin{proof}
Let \(\tau = (t_1,\vec t_2 ,\dots , \vec t_n)\).
If \(L(\tau) = 1\), we have \(\tau'' = (f, \vec t_1' , \vec t_2, \dots,\vec t_n)\) for some \(f \in \Sigma\) and \(\vec t_1'\) such that \(t_1 = f\vec t_1'\).
Then, \(\tau'\) can be written as either (i) \((t_{1,k}' , \vec t_2 , \dots , \vec t_n)\) for some \(k=1,\dots,l\) (\(\vec t_1' = \ab<t_{1,1}',\dots,t_{1,l}>\)), (ii) \((f,\vec t_1',\vec t_2,\dots,\vec t_{n-1})\), or (iii) \(d_j(\tau'')\) for some \(j=2,\dots,n-1\).
For (i) and (ii), \(T(\tau')\) is a proper generalized subterm of \(T(\tau)\).
For (iii), we can see that \(T(\tau)\to_R^+ T(\tau')\) if \(T(\tau) \neq T(\tau')\), and 2 in \cref{def:wford} holds if \(T(\tau) = T(\tau')\).

If \(L(\tau) > 1\), let \(\vec t_{L(\tau)} = \vec t_{L(\tau)}' \vec t_{L(\tau)}''\) be the \((\vec t_{1}\dots\vec t_{L(\tau)-1})\)-decomposition.
Again, \(\tau'\) can be written as either
(i') \((t_{2,k},\vec t_3,\dots,\vec t_n)\) for some \(k=1,\dots,m\) where \(\vec t_2 = \ab<t_{2,1},\dots,t_{2,m}>\),
(ii') \((t_1,\vec t_2,\dots ,\vec t_{n-1})\), or
(iii') \(d_j(\tau'')\) for some \(j=1,\dots,n-1\), \(j \neq L(\tau)\),
and, in case (i') or (ii'), \(T(\tau')\) is a proper subterm of \(T(\tau)\).
So, consider (iii').

For \(\tau'' = (s_1,\vec s_2,\dots ,\vec s_{n+1})\), if \(\vec s_j\vec s_{j+1}\) is \(R\)-reducible, we have \(T(\tau) \to^+_R T(\tau')\).
Suppose that \(\vec s_j\vec s_{j+1}\) is \(R\)-irreducible.
If \(j \ge L(\tau)+1\), we have \(\vec s_j = \vec t_{j-1}\), \(\vec s_{j+1} = \vec t_{j}\), and
\[d_j(\tau'') = (t_1, \vec t_2, \dots , \vec t_{L(\tau)}',\vec t_{L(\tau)}'' , \vec t_{L(\tau)+1} , \dots , \vec t_{j-1} \vec t_{j} , \dots \vec t_n).\]
Then, 3(b) holds with \(i=L(\tau)-1\).

If \(j < L(\tau)\), we have \(\vec s_j = \vec t_j\), \(\vec s_{j+1} = \vec t_{j+1}\), and \(d_j(\tau'') = (t_1, \vec t_2, \dots , \vec t_j \vec t_{j+1} , \dots , \vec t_{L(\tau)}',\vec t_{L(\tau)}'' , \vec t_{L(\tau)+1} , \dots ,\vec t_n)\).
Then, the definition of chains ensures that 3(a) holds with \(i = j-1\).
\end{proof}

By lemmas above, we have
\begin{theorem}
\(\cM\) is a Morse matching.
\end{theorem}

Therefore, if \(R\) is a complete TRS,
\[
\dots \to B_2^\cM \xrightarrow{\delta_2^\cM} B_1^\cM \xrightarrow{\delta_1^\cM} B_0^\cM \to \kahlerdiffs_\varlaw \to 0
\]
is a free resolution of \(\kahlerdiffs_\varlaw\) and \(B_n^\cM\) is the free \(\cU_\varlaw\)-module generated by \(n\)-chains.
Moreover, each \(B_n^\cM\) is a finitely generated \(\cU_\varlaw\)-module if \(S,\varsign,R\) are finite sets.

\begin{remark}
Suppose that a 3-cell \(\tau = (f,\vec t_2,\vec t_3)\) is critical.
Then, since \((f,\vec t_2)\) is also critical, \(f\vec t_2 = l_1 \in R_\mathrm{L}\).
Also, \(\bfp(f\vec t_2\vec t_3) = (p, l_2)\), so \((f\vec t_2\vec t_3)|_p = l_2\vec s\) for some \(\vec s\).
We can check that \(l_1,l_2\) provides a critical pair.
(It is a fun coincidence that critical cells correspond to critical pairs, even though the two uses of ``critical'' originated in independent fields, namely Morse theory and rewriting theory.)
\end{remark}

\section{Morse inequalities}
For a right \(\cU\)-module \(M\), let \(H_n(\varlaw; M) = \Tor_n^{\cU_\varlaw}(M, \kahlerdiffs_\varlaw)\).
In this section, we will define an appropriate right module \(M\) to obtain Morse inequalities for \(\varlaw\).

If an \(\varsorts\)-sorted Lawvere theory \(\varlaw\) is presented by a TRS \((\varsign,R)\),
the ringoid \(\cR_\varlaw^{\varobj}\) has the presentation \(\ab(Q^{\vec X}_{(\Sigma,R)}, U^{\vec X}_{(\Sigma,R)})\) defined as follows.
The quiver \(Q^{\vec X}_{(\Sigma,R)}\) is the subquiver of \(Q^{\vec X}\) defined in \cref{sec:bar_resolution} such that the set of vertices of \(Q^{\vec X}_{(\Sigma,R)}\) is the same as that of \(Q^{\vec X}\) and \(Q^{\vec X}_{(\Sigma,R)}\) has only edges of the form \(\partial_i(f)_{\vec \sigma} = \partial_i(x_1,\dots,x_n\mid_R f(x_1,\dots,x_n))\) for \(f \in \Sigma\).
The set of relations \(U^{\vec X}_{(\Sigma,R)}\) contains
\begin{equation}\label{eqn:relation_rew}
\kappa_i(\Gamma \mid_\emptyset l)_{\vec\sigma} = \kappa_i(\Gamma \mid_\emptyset r)_{\vec\sigma}
\end{equation}
for each rule \(\Gamma \mid_\emptyset l \to r \in R\) and integer \(i\), where
\(\kappa_i(\Gamma \mid_\emptyset t)_{\vec \sigma} \in \bbZ Q(\sigma_i, t\vec \sigma)\) is defined as
\[
\kappa_i(x_1,\dots,x_m \mid_\emptyset x_j)_{\vec \sigma} =
\begin{cases}
\Id_{\sigma_j} & (i=j)\\
0 & (i \neq j)
\end{cases}
\]
\[
\kappa_i(\Gamma \mid_\emptyset f(t_1,\dots,t_n))_{\vec \sigma} =
\sum_{j=1}^n \partial_j(f)_{\vec t \vec \sigma} \kappa_i(\Gamma \mid_\emptyset t_i)_{\vec \sigma}.
\]

\begin{lemma}
\begin{enumerate}
	\item \(\kappa_i(\Gamma' \mid_\emptyset t\ab<t_1',\dots,t_n'>)_{\vec\sigma} = \sum_j\kappa_j(\Gamma\mid_\emptyset t)_{\ab<(\Gamma' \mid_R t'_1)\vec\sigma,\dots,(\Gamma'\mid_R t'_n)\vec\sigma>}\kappa_i(\Gamma'\mid_\emptyset t_j')_{\vec\sigma}\),
	\item For any rule \(\Gamma' \mid_\emptyset l \to r\),
\begin{align*}
&\kappa_i(\Gamma'' \mid_\emptyset s\ab<t_1,\dots,t_n,l>\vec u)_{\vec \sigma} - \kappa_i(\Gamma'' \mid_\emptyset s\ab<t_1,\dots,t_n,r>\vec u)_{\vec\sigma}\\
&= \kappa_{n+1}(\Gamma \mid_\emptyset s)_{\ab<(\Gamma'\mid_R t_1),\dots,(\Gamma' \mid_R t_n),(\Gamma'\mid_R l)>\vec u\vec\sigma}
\sum_j \ab( \kappa_j(\Gamma' \mid_\emptyset l)_{(\Gamma'' \mid_R \vec u)\vec\sigma} - \kappa_j(\Gamma' \mid_\emptyset r)_{(\Gamma''\mid_R \vec u)\vec\sigma} )
\kappa_i(\Gamma'' \mid_\emptyset u_j)_{\vec\sigma}
\end{align*}
\end{enumerate}
\end{lemma}
\begin{proof}
1. By induction on the structure of \(t\).

2.
Using 1, we have
\begin{align*}
&\kappa_i(\Gamma'' \mid_\emptyset s\ab<t_1,\dots,t_n,l>\vec u)_{\vec \sigma} - \kappa_i(\Gamma'' \mid_\emptyset s\ab<t_1,\dots,t_n,r>\vec u)_{\vec\sigma}\\
&= \sum_{k=1}^n \kappa_j(\Gamma \mid_\emptyset s)_{\ab<\tau_1,\dots,\tau_n,(\Gamma' \mid_R l)>\vec\rho\vec\sigma} \kappa_i(\Gamma'' \mid_\emptyset t_k \vec u)_{\vec\sigma}
+ \kappa_{n+1}(\Gamma \mid_\emptyset s)_{\ab<\tau_1,\dots,\tau_n,(\Gamma' \mid_R l)>\vec\rho\vec\sigma} \kappa_i(\Gamma'' \mid_\emptyset l\vec u)_{\vec\sigma}\\
&\quad - \sum_{k=1}^n \kappa_j(\Gamma \mid_\emptyset s)_{\ab<\tau_1,\dots,\tau_n,(\Gamma' \mid_R r)>\vec \rho\vec\sigma	} \kappa_i(\Gamma'' \mid_\emptyset t_j \vec u)_{\vec\sigma}
- \kappa_{n+1}(\Gamma \mid_\emptyset s)_{\ab<\tau_1,\dots,\tau_n,(\Gamma' \mid_R r)>\vec \rho\vec\sigma} \kappa_i(\Gamma'' \mid_\emptyset r\vec u)_{\vec\sigma}\\
&= \kappa_{n+1}(\Gamma \mid_\emptyset s)_{\ab<\tau_1,\dots,\tau_n,(\Gamma' \mid_R l)>\vec \rho\vec\sigma}
\ab(\kappa_i(\Gamma'' \mid_\emptyset l\vec u)_{\vec\sigma} - \kappa_i(\Gamma'' \mid_\emptyset r\vec u)_{\vec\sigma})\\
&= \kappa_{n+1}(\Gamma \mid_\emptyset s)_{\ab<\tau_1,\dots,\tau_n,(\Gamma' \mid_R l)>\vec \rho\vec\sigma}
\sum_j \ab( \kappa_j(\Gamma' \mid_\emptyset l)_{\vec \rho\vec\sigma} - \kappa_j(\Gamma' \mid_\emptyset r)_{\vec \rho\vec\sigma} )
\kappa_i(\Gamma'' \mid_\emptyset u_j)_{\vec\sigma}
\end{align*}
where \(\tau_i = (\Gamma'\mid_R t_i)\), \(\vec\rho = (\Gamma'' \mid_R \vec u)\).
\end{proof}

\begin{lemma}
\(\ab(Q^{\vec X}_{(\Sigma,R)}, U^{\vec X}_{(\Sigma,R)})\) is a presentation of \(\cR_\varlaw^{\varobj}\).
\end{lemma}
\begin{proof}
Consider the inclusion \(\iota : \bbZ Q^{\vec X}_{(\Sigma,R)} \hookrightarrow \bbZ Q^{\vec X}\).
We can show that
\(\iota(\kappa_i(\Gamma \mid_\emptyset t)_{\vec\sigma}) = \partial_i(\Gamma \mid_R t)_{\vec\sigma}\) modulo \(U^{\vec  X}\) for any \(\Gamma \mid_\emptyset t\), by induction on the structure of \(t\).
So, for any rule \(\Gamma \mid_\emptyset l \to r \in R\), we have \(\iota(\kappa_i(\Gamma \mid_\emptyset l)_{\vec\sigma}) = \iota(\kappa_i(\Gamma \mid_\emptyset r)_{\vec\sigma})\) modulo \(U^{\vec X}\).
Also, every edge \(\partial_i(\omega)_{\vec \sigma}\) of \(Q^{\vec X}\)
is equal to \(\iota(\kappa_i(\Gamma \mid_\emptyset t)_{\vec\sigma})\) modulo \(U^{\vec X}\) if \(\omega = (\Gamma \mid_R t)\).
That is, we have an additive functor \(\mu : \bbZ Q^{\vec X} \to \bbZ Q^{\vec X}_{(\Sigma,R)}\) such that \(\bbZ Q^{\vec X} \xrightarrow{\mu} \bbZ Q^{\vec X}_{(\Sigma,R)} \xhookrightarrow{\iota} \bbZ Q^{\vec X}\) is the identity functor, and \(\mu\) depends on the choice of \((\Gamma \mid_\emptyset t)\) for each \(\omega = (\Gamma \mid_R t)\).
We fix such a choice arbitrarily and show
\begin{equation*}
\mu(\partial_i(\omega\ab<\omega'_1,\dots,\omega'_n>)_{\vec\sigma}) = \mu\ab(\sum_{j}\partial_j(\omega)_{\ab<\omega_1'\vec\sigma,\dots,\omega_n'\vec\sigma>}\partial_i(\omega'_i)_{\vec\sigma}) \mod U^{\vec X}_{(\Sigma,R)}
\end{equation*}
for any \(\omega\), \(\omega_1',\dots,\omega_n'\), \(\vec\sigma\), and \(i\).
Suppose that \((\Gamma \mid_\emptyset t)\), \((\Gamma' \mid_\emptyset t'_i)\), \((\Gamma' \mid_\emptyset s)\) are chosen for \((\Gamma \mid_R t) = \omega\), \((\Gamma' \mid_R t'_i) = \omega_i'\), \((\Gamma' \mid_R s) = \omega\ab<\omega'_1,\dots,\omega'_n>\).
Then, we have
\begin{align*}
\mu(\partial_i(\omega\ab<\omega'_1,\dots,\omega'_n>)_{\vec\sigma})
&= \kappa_i(\Gamma' \mid_\emptyset s)_{\vec\sigma},\\
\mu\ab(\sum_j\partial_j(\omega)_{\ab<\omega_1'\vec\sigma,\dots,\omega_n'\vec\sigma>}\partial_i(\omega'_i)_{\vec\sigma})
&= \sum_j \kappa_j(\Gamma \mid_\emptyset t)_{\ab<\omega_1'\vec\sigma,\dots,\omega_n'\vec\sigma>} \kappa_i(\Gamma' \mid_\emptyset t'_i)_{\vec\sigma}\\
&= \kappa_i(\Gamma' \mid_\emptyset t\ab<t_1',\dots,t_n'>)_{\vec\sigma}.
\end{align*}
So, it suffices to show
\begin{equation}\label{eqn:mu_rel}
\kappa_i(\Gamma' \mid_\emptyset s)_{\vec\sigma} = \kappa_i(\Gamma' \mid_\emptyset s')_{\vec\sigma} \mod U^{\vec X}_{(\Sigma,R)}
\end{equation}
for any \(s,s'\) with \(s \approx_R s'\).
Let \(\leftrightarrow_R\) be the symmetric closure of \(\to_R\).
Then, \(\approx_R\) is the reflexive transitive closure of \(\leftrightarrow_R\).
Since \(s \approx_R s'\),
there are terms \(s_0,\dots,s_N\) such that
\(s_0 = s \leftrightarrow_R s_1 \leftrightarrow_R \dots \leftrightarrow_R s_N = s'\).
We show \eqref{eqn:mu_rel} by induction on \(N\).
If \(N=0\), \eqref{eqn:mu_rel} is obvious since \(s = s'\).
Suppose \(N > 0\) and \(\kappa_i(\Gamma_0 \mid_\emptyset s)_{\vec\sigma} = \kappa_i(\Gamma_0 \mid_\emptyset s_{N-1})_{\vec\sigma}\).
Also, suppose \(s_{N-1} \to_R s_N = s'\).
The case \(s_N \to_R s_{N-1}\) can be shown similarly.
Then, \(s_{N-1}\) and \(s'\) can be written as \(s_{N-1} = u\ab<v_1,\dots,v_m,l>\vec w\) and \(s' = u\ab<v_1,\dots,v_m,r>\vec w\) for some rule \(\Gamma_1 \vdash l \to r \in R\) and term-in-contexts \(\Gamma_2 \vdash u\), \(\Gamma_1 \vdash v_1,\dots,v_m\), \(\Gamma_0\vdash \vec w\).
By the previous lemma, we have
\begin{align*}
&\kappa_i(\Gamma_0 \mid_\emptyset s_{N-1})_{\vec\sigma} - \kappa_i(\Gamma' \mid_\emptyset s')_{\vec\sigma}\\
&= \kappa_{m+1}(\Gamma \mid_\emptyset u)_{\ab<\tau_1,\dots,\tau_m,(\Gamma_1\mid_\emptyset l)>\vec\rho\vec\sigma}
\sum_j \ab( \kappa_j(\Gamma'' \mid_\emptyset l)_{\vec\rho\vec\sigma} - \kappa_j(\Gamma'' \mid_\emptyset r)_{\vec \rho\vec\sigma} )
\kappa_i(\Gamma' \mid_\emptyset w_j)_{\vec\sigma} \in \ab(U^{\vec X}_{(\Sigma,R)}).
\end{align*}
where \(\tau_i = (\Gamma_1\mid_\emptyset v_i)\) and \(\vec\rho = (\Gamma_0 \mid_\emptyset \vec w)\).
\end{proof}

We call coefficients of form \(\partial_{i_1}(f_1)_{\vec\sigma_1}\dots \partial_{i_m}(f_m)_{\vec\sigma_m}\) \emph{monomials}.

For a term \(t\) and a variable \(x\), let \(\#_x t\) be the number of occurrences of the variable \(x\) in \(t\).
The following is easy to show.
\begin{lemma}
For any \((x_1,\dots,x_n\mid_\emptyset t)\),
integer \(i= 1,\dots,n\), and morphism \(\vec\sigma\) in \(\varlaw\),
\(\kappa_i(x_1,\dots,x_n \mid_\emptyset t)_{\vec\sigma}\) is the sum of \(\#_{x_i}t\) monomials.
\end{lemma}

\begin{definition}
The \emph{degree} of \(R\) is \(\gcd\{\#_xl - \#_x r \mid l \to r \in R, x \in V\}\).
\end{definition}

Suppose that \(R\) has degree \(d\).

\begin{lemma}
If there are \(k\) monomials \(D_1,\dots,D_k\) and \(k'\) monomials \(D_1',\dots,D_{k'}'\) such that \(D_1 + \dots + D_k - D_1' - \dots - D_{k'}' \in \ab(U^{\vec X}_{(\Sigma,R)})\), then \(k - k' \in d\bbZ\).
\end{lemma}
\begin{proof}
Any \(D_1 + \dots + D_k - D_1' - \dots - D_{k'}' \in \ab(U^{\vec X}_{(\Sigma,R)})\) can be written as
\[
\sum_j \pm D_j'' \ab(\kappa_{i_j}(\Gamma_j \mid_\emptyset l_j )_{\vec \sigma_j} - \kappa_{i_j}(\Gamma_j \mid_\emptyset r_j)_{\vec \sigma_j})D'''_j
\]
for some monomials \(D_j'',D'''_j\), rules \(\Gamma_j \mid_\emptyset l_j \to r_j \in R\), integers \(i_j\), and morphisms \(\vec \sigma_j\) in \(\varlaw\).
Therefore, we have \(k - k' = \sum_j \pm (\#_{i_j} l_j - \#_{i_j}r_j) \in d\bbZ\).
\end{proof}

We define a right \(\cU_\varlaw\)-module \(\cZ_d\) as follows.
\begin{lemma}
Let \(\cZ_d(\omega) = \bbZ/d\bbZ\) for any \(\omega \in \Ob(\cU_\varlaw)\),
and let
\(\cZ_d\ab( D\vec\alpha^* ) = \Id_{\bbZ/d\bbZ}\) for any monomial \(D\).
Then, \(\cZ_d\) extends to a unique right \(\cU_\varlaw\)-module.
\end{lemma}
\begin{proof}
We show that, for any morphism \(r = \sum_i \epsilon_i D_i \vec\alpha_i\) (\(\epsilon_i = \pm 1\)) in \(\cU_\varlaw\), if \(r=0\), then \(\sum_i \epsilon_i \Id_{\bbZ/d\bbZ} = 0\), i.e., \(\sum_i \epsilon_i \in d\bbZ\).
We can rewrite \(r\) to the form \(r = \sum_j D_j'\vec\beta^*_j\) where
\[
D_j' = D_{j,1}^+ + \dots + D_{j,k_j}^+ - D_{j,1}^- - \dots - D_{j,k'_j}^-,
\]
for monomials \(D_{j,1}^+,\dots,D_{j,k_j}^+,D_{j,1}^-,\dots,D_{j,k'_j}^-\) and
\(\vec\beta_j\neq \vec\beta_{j'}\) for any \(j \neq j'\).
If \(r = 0\), we have \(D_j \in \ab(U^{\vec X}_{(\Sigma,R)})\) for every \(j\), and it implies \(\sum_i \epsilon_i = \sum_j k_j - k'_j \in d\bbZ\) by the previous lemma.
\end{proof}

\begin{lemma}
For any free left \(\cU_\varlaw\)-module \(M\) generated by \(\mathcal{X} : \Ob(\cU_\varlaw) \to \catset\),
\(\cZ_d \otimes_{\cU_\varlaw} M\) is the free \(\bbZ/d\bbZ\)-module generated by \(\coprod \mathcal{X}\).
\end{lemma}
\begin{proof}
It is not difficult to show that the map \(1 \otimes rx \mapsto x\) gives an isomorphism.
\end{proof}

\begin{theorem}
Let \(R\) be a reduced and complete TRS with degree \(d\) such that \(d=0\) or \(d\) is prime. Then, we have
\begin{description}
	\item[(Weak Morse Inequality)] \(\#\Cr_n(\cM_{\varsign,R}) \ge s( H_n(\varlaw; \cZ_d)\))
	\item[(Strong Morse Inequality)] 
\(\sum_{i=0}^n (-1)^{n-i} \#\Cr_i(\cM_{\varsign,R}) \ge s(H_n(\varlaw; \cZ_d)) + \sum_{i=0}^{n-1} (-1)^{n-i} \mathrm{rank} H_i(\varlaw; \cZ_d)\)
\end{description}
where \(s(M)\) is the minimum number of generators of a \(\bbZ/d\bbZ\)-module \(M\).
\end{theorem}
\begin{proof}
Let \(C_n = \cZ_d \otimes_{\cU_\varlaw} B_n^{\cM_{\Sigma,R}} \).
By the previous lemma, each \(C_n\) is a free \(\bbZ/d\bbZ\)-module generated by \(\Cr_n(\cM_{\Sigma,R})\).
Then, we have \(\#\Cr_n(\cM_{\Sigma,R}) = \rank(C_n) \ge \rank(\ker(C_n \to C_{n-1})) \ge s(H_n(\varlaw; \cZ_d))\) since \(\bbZ/d\bbZ\) is a PID.
Also, by a basic fact on Euler characteristics, we have
\[
\sum_{i=0}^{n-1} (-1)^{n} \#\Cr_i(\cM_{\varsign,R})
=\rank(\ker(C_n \to C_{n-1})) + \sum_{i=0}^{n-1} (-1)^{n-i} \rank(H_i(\varlaw; \cZ_d)),
\]
and since \(\rank(\ker(C_n \to C_{n-1})) \ge s(H_n(\varlaw; \cZ_d))\), we obtain the second inequality.
\end{proof}

\bibliographystyle{plain}
\bibliography{lawvere_morse}

\end{document}